\documentclass{amsart}
\usepackage{amsmath}
\usepackage{amsfonts}
\usepackage{amsthm, upref}
\usepackage{graphicx}

\newcommand{\e}{\mathrm{E}}
\newcommand{\mcal}[1]{\mathcal{#1}}
\newcommand{\dx}{\mathrm{d}}

\newcommand{\abs}[1]{\left\lvert #1\right\rvert}

\newcommand{\lone}{\mathbf{L}^{1}}

\newcommand{\rr}{\mathbb{R}}
\renewcommand{\mcal}[1]{\mathcal{#1}}

\newcommand{\pp}{\varepsilon}

\newcommand{\flr}[1]{\lfloor #1 \rfloor}

\newcommand{\comment}[1]{}

\begin{document}

\theoremstyle{plain}
\newtheorem{thm}{Theorem}
\newtheorem{lem}[thm]{Lemma}
\newtheorem{prop}[thm]{Proposition}
\newtheorem{cor}[thm]{Corollary}

\theoremstyle{definition}
\newtheorem{defn}{Definition}
\newtheorem{asmp}{Assumption}
\newtheorem{notn}{Notation}
\newtheorem{prb}{Problem}

\theoremstyle{remark}
\newtheorem{rmk}{Remark}
\newtheorem{exm}{Example}
\newtheorem{clm}{Claim}

\title[Rank dependent Brownian motions]{A phase transition behavior for Brownian motions interacting through their ranks}

\author{Sourav Chatterjee}
\address{367 Evans Hall \# 3860\\
Univ. of California at Berkeley\\
Berkeley, CA 94720-3860}
\email{sourav@stat.berkeley.edu}

\author{Soumik Pal}
\address{506 Malott Hall\\ Cornell University\\ Ithaca, NY 14853}
\email{soumik@math.cornell.edu}



\thanks{Sourav Chatterjee's research is partially supported by N.S.F. grant DMS-0707054.}
\thanks{Soumik Pal's research is partially supported by N.S.F. grant DMS-0306194 to the probability group at Cornell.}


\date{\today}
\maketitle

\begin{abstract}
Consider a time-varying collection of $n$ points on the positive real axis, modeled as Exponentials of $n$ Brownian motions whose drift vector at every time point is determined by the relative ranks of the coordinate processes at that time. If at each time point we divide the points by their sum, under suitable assumptions the rescaled point process converges to a stationary distribution (depending on $n$ and the vector of drifts) as time goes to infinity. This stationary distribution can be exactly computed using a recent result of Pal and Pitman. The model and the rescaled point process are both central objects of study in models of equity markets introduced by Banner, Fernholz, and Karatzas. In this paper, we look at the behavior of this point process under the stationary measure as $n$ tends to infinity. Under a certain `continuity at the edge' condition on the drifts, we show that one of the following must happen: either (i) all points converge to $0$, or (ii) the maximum goes to $1$ and the rest go to $0$, or (iii) the processes converge in law to a non-trivial Poisson-Dirichlet distribution. The underlying idea of the proof is inspired by Talagrand's analysis of the low temperature phase of Derrida's Random Energy Model of spin glasses. The main result establishes a universality property for the BFK models and aids in explicit asymptotic computations using known results about the Poisson-Dirichlet law.

\end{abstract}


\section{Introduction} 

Let $\delta_1,\delta_2,\ldots,\delta_n$ be $n$ real constants, and consider the following system of stochastic differential equations: 
\begin{equation}\label{ranksde}
dX_i(t) = \sum_{j=1}^n \delta_j 1\left(X_i(t)=X_{(j)}(t)\right)dt + dW_i(t),\quad i=1,2,\ldots,n.
\end{equation}
Here $X_{(1)}(t) \ge X_{(2)}(t) \ge \ldots \ge X_{(n)}(t)$ are the coordinates of the process in the decreasing order, and $W=(W_1,W_2,\ldots,W_n)$ is an $n$-dimensional Brownian motion. The SDE models the movement of $n$ particles as interacting Brownian motions such that at every time point, if we order the positions of the particles, then the $i$th ranked particle from the top gets a drift $\delta_i$ for $i=1,2,\ldots,n$. As time evolves, the Brownian motions switch ranks and drifts, and hence their motion is determined by such time dependent interactions. Such processes (or close relatives) have been considered recently in several articles. We consider three major veins in the literature where these processes make their appearance.
\medskip

\noindent \textbf{1. Stochastic Portfolio Theory.} A detailed study of the solutions of SDE \eqref{ranksde} in general was taken up in a paper by Banner, Fernholz, and Karatzas (BFK) \cite{atlasmodel}. These authors actually consider a more general class of SDEs than \eqref{ranksde} in which the drifts as well as the volatilities of the different Brownian motions depend on their ranks. Using a delicate and elegant analysis of the local time of intersections of the different Brownian motions, the authors were able to demonstrate various ergodic properties of such processes. An important object in their study is the vector of \emph{spacings} in the ordered process:
\begin{equation}\label{spacings}
Y_i(t) := X_{(i)}(t) - X_{(i+1)}(t), \quad 1\le i \le n-1.
\end{equation} 
Under suitable sufficient conditions the authors proved that the marginal distributions of the spacings converge to Exponential laws and left the question of joint convergence open. The joint law was identified and the BFK conditions were shown to be necessary and sufficient in Pal and Pitman \cite{palpitman} (see Theorem 
\ref{theoremN} below for the statement of the Pal-Pitman result). \medskip

\noindent{\bf 2. Burgers', McKean-Vlasov, and the granular media equations.} Sznitman in his articles on the propagation of chaos, \cite{sznitman86} and \cite{sznitman91}, considers $n$ Brownian motions arranged in decreasing order. He proves that the ordered process can be thought of as a multidimensional Brownian motion reflected in the orthant $\{ x_1 \ge x_2\ldots \ge x_n\}$ and discusses its relation with the solution of the Burgers' equation. For more details on reflected Brownian motions in an orthant, see Varadhan and Williams \cite{varwilliams} and Williams \cite{williams87r}. For more on stochastic particle system methods in analyzing Burgers' equation, see Calderoni and Pulvirenti \cite{calpul}. Note that the drifts in SDE \eqref{ranksde} at any time point are a function of the empirical cumulative distribution function induced by the $n$ Brownian particles on the line. In stochastic models related to McKean-Vlasov and the granular media PDE, the empirical CDF, for large $n$, is an approximation to a single probability measure (depending on t). Thus,  \eqref{ranksde} and similar systems often appear as approximations to SDEs whose drifts and volatilities depend on the marginal distribution of the particles. For example, in the context of McKean-Vlasov type equations, see Bossy and Talay \cite{bostal} and the recent article by Bolley, Guillin, and Villani \cite{bolguivil}. For the granular media equation, see the article by Cattiaux, Guillin, and Malrieu \cite{catgmal} which also lucidly describes the physics behind the modeling.  Also see Jourdain \cite{jourdain00} and Malrieu \cite{malrieu01} for an integrated approach to the McKean-Vlasov and granular media differential equations and related interacting diffusions. From similar perspectives Jourdain and Malrieu \cite{joumal} considers SDE \eqref{ranksde} with an increasing sequence of $\delta_i$'s and establishes joint convergence of the spacing system \eqref{spacings} to independent Exponentials as $t$ goes to infinity. 
\medskip

\noindent{\bf 3. Interacting particle systems and queueing networks.} The ordered particle system of Brownian motions had been studied even earlier by Harris in \cite{harris65}. Harris considers a countable collection of ordered Brownian motions with no drifts (i.e. $\delta_i=0$), and analyzes invariant distributions for the spacings between the ordered processes. Arratia in \cite{arratia85} (see also \cite{arratia83}) observes that instead of Brownian motion, if one considers the exclusion process associated with the nearest neighbor random walk on $\mathbb{Z}$, then the corresponding stationary system of spacings between particles can be interpreted as a finite or infinite series of queues. We direct the reader to the seminal articles by Harrison and his coauthors for a background on systems of Brownian queues: \cite{harrison73}, \cite{harrison00}, \cite{harvan97}, \cite{harwil87s}, and \cite{harrison78}. Baryshnikov \cite{barysh01} establishes the connections between Brownian queues and GUE random matrices. For more references in this direction, see Pal and Pitman \cite{palpitman}. Also, see the related discrete time models in statistical physics studied by Ruzmaikina and Aizenman \cite{ruzaizenman}, and Arguin and Aizenmann \cite{arguinaizenman}. In particular see the article by Arguin \cite{arguinPD} where he connects the Poisson-Dirichlet's with the competing particle models.
\medskip
 
The article by Pal and Pitman \cite{palpitman} considers solutions of SDE \eqref{ranksde} as their central object of study. The number $n$ can either be finite or countably infinite. They show the existence of a weak solution for SDE \eqref{ranksde} and the uniqueness in law (see Lemma 3 in \cite{palpitman}). One of the major results in the paper  is a necessary and sufficient condition on the sequence $\{\delta_1,\ldots,\delta_n\}$ such that the law of the spacing vector $(Y_1(t),Y_2(t),\ldots,Y_{n-1}(t))$ converges in total variation to a unique stationary distribution as time goes to infinity. We would like to caution that \cite{palpitman} uses the order statistic notation in which $X_{(1)} \le X_{(2)} \le \ldots$ while describing SDE \eqref{ranksde}, which is the reverse of the notation we found to be the most suitable for this paper. We state part of the result below following our notation in \eqref{ranksde}.

\begin{thm}[Pal and Pitman \cite{palpitman}, Theorem 4]\label{theoremN} 
Consider the SDE in \eqref{ranksde}. For $1 \le k \le n$, let
\begin{equation}\label{alphak}
\alpha_k:= \sum_{i = 1}^k ( \bar{\delta}_n - \delta_i   )  
\end{equation}
where $\bar{\delta}_n$ is the average drift $(\delta_1+\delta_2+\ldots+\delta_n)/n$. 

For each fixed initial distribution of the $n$ particle system with
drifts $\delta_i$, the collection of laws of 
$X_{(1)}(t) - X_{(n)}(t)$ for $t \ge 0$ is tight if
and only if 
\begin{equation}\label{conal}
\alpha _k > 0  \mbox{ for all } 1 \le k \le n-1,
\end{equation}
in which case the following result holds:

The distribution of the spacings system 
\[
\left(Y_j(t)=X_{(j)}(t) - X_{(j+1)}(t),\quad1\le j\le n-1\right)
\]
at time $t$ converges in total variation norm as $t \rightarrow \infty$ to a unique stationary distribution which is that of independent Exponential variables $Y_j$ with rates $2\alpha_j$, $1 \le j \le n-1$.
Moreover, the spacings system is reversible at equilibrium.
\end{thm}

\noindent{\it Remark.} The necessity of condition \eqref{conal} is hinted by the following localization argument: for every $k$, one divides the collection of $n$ particles as the top $k$ particles and the bottom $n-k$ particles. If the entire collection has to remain together, the average drift of the top $k$ particles must be less than than the average drift of the bottom $n-k$. A simple computation then implies \eqref{conal}. 
\medskip

Now suppose we have a triangular array of drifts $(\delta_i(n),\;1\le i\le n,\; n\in \mathbb{N})$, such that for every $n$ the vector $(\delta_1(n),\delta_2(n),\ldots,\delta_n(n))$ satisfies condition~\eqref{conal}. 
Note that the market weights \eqref{mkwts} are purely a function of the spacings, and thus have an induced stationary measure for every $n$. Thinking of the market weights under this stationary measure as a random point process in $[0,1]$, we are interested in the limit law of this point process as we take $n$ to infinity. It turns out that any non-trivial limit law is a Poisson-Dirichlet distribution under `universal conditions', and exhibits phase transitions. The results are summarized in the following theorem. For the definition of the Poisson-Dirichlet family of distributions, see Section \ref{pp}.


\begin{thm}\label{mainthmdrift}
Suppose for every $n\in \mathbb{N}$ we have a sequence of constants $(\delta_i(n),\;i=1,2,\ldots,n)$, and consider $n$ interacting particles satisfying the following SDE:
\[
d X_i(t) = \sum_{j=1}^n \delta_j(n)1\left(X_i(t)=X_{(j)}(t)\right)dt + dW_i(t),\quad i=1,2,\ldots,n.
\]
Here $X_{(1)}\ge X_{(2)}\ge\ldots\ge X_{(n)}$ are the ordered processes, and $W$ is an $n$-dimensional Brownian motion. 

Suppose for every $n\in \mathbb{N}$, the drift sequence $(\delta_i(n),\;1\le i\le n)$ satisfies condition~\eqref{conal}. Then by Theorem \ref{theoremN}, the decreasing point process
\begin{equation}\label{pprev}
\biggl(\frac{e^{X_{(j)}(t)}}{\sum_{i=1}^n e^{X_{(i)}(t)}}\biggr)_{j=1,2,\ldots,n}
\end{equation}
converges in distribution as $t\rightarrow \infty$. Let $(\mu_{(1)}(n),\ldots,\mu_{(n)}(n))$ denote a point process (of the market weights) drawn from the stationary law.
Now let $\bar{\delta}(n)= (\delta_1(n)+\ldots+ \delta_n(n))/n$ and assume that there exists  $\eta \in [0,\infty]$ such that for every fixed $i\ge 1$,
\begin{eqnarray}
\lim_{n\rightarrow \infty}\left(\bar{\delta}(n) - \delta_i(n)\right)&=&\eta.\label{conddel2}
\end{eqnarray} 
Additionally, assume that
\begin{eqnarray}
\limsup_{n\rightarrow \infty}\max_{1\le i\le n}\left(\bar{\delta}(n) - \delta_i(n) \right)&\le& \eta.\label{conddel}
\end{eqnarray}
Then, as $n$ tends to infinity,
\begin{itemize}

\item[(i)] if $\eta \in (0, 1/2)$, the point process generated by the market weights converges weakly to a Poisson-Dirichlet process with parameter $2\eta$.

\item[(ii)] If $\eta \in [1/2,\infty]$, all the market weights converge to zero in probability.

\item[(iii)] If $\eta=0$, the market weights converge in probability to the sequence $(1,0,0,\ldots)$, i.e., the largest weight goes to one in probability while the rest go to zero.
\end{itemize}
In particular, $\mu_1(n)$ converges weakly to a non-trivial distribution if and only if $\eta\in (0, 1/2)$.
\end{thm}

The conditions \eqref{conddel2} and \eqref{conddel} admit natural interpretations. Condition \eqref{conddel2} is the `continuity at the edge' condition mentioned in the abstract. It means that the top few particles get almost the same drift. This is a weaker version of a general continuity condition that would demand that whenever two particles have almost the same rank, they would get almost the same drifts. The condition \eqref{conddel} can be understood as saying that the highest ranked particles get the minimum drifts, in a limiting sense. Although the conditions seem to be reasonable, it is still surprising that the conditions are also sharp, in the sense that Poisson-Dirichlet convergence may fail if either \eqref{conddel2} or \eqref{conddel} fails to hold as shown by the counterexamples in  Section \ref{examples}. Finally we would like to mention that the two degenerate limits can be seen as limits of the Poisson-Dirichlet laws as the parameter $2\eta$ tends to zero or one.

Now, Theorem \ref{mainthmdrift} says that $\mu_1(n) \rightarrow 0$ in probability if $\eta \ge 1/2$. It is natural to wonder about the rate of convergence. It turns out that some exact information may be extracted after putting in an additional `Lipschitz' assumption on the drifts, even at the `critical value' $\eta = 1/2$.

\begin{thm}\label{phasegrav}
Assume the same setup as in Theorem \ref{mainthmdrift}. Let
\[
\eta = \lim_{n\rightarrow \infty} (\bar{\delta}(n) - \delta_1(n)),
\]
and instead of \eqref{conddel2} and \eqref{conddel}, assume that there exists constants $C > 0$ and $0<\gamma < 1$ such that whenever $i \le \gamma n$, we have
\begin{equation}\label{totlips}
\abs{\delta_i(n) - \delta_1(n)} \le \frac{C(i-1)}{n}.
\end{equation}
Then, if $\eta > 1/2$, we have 
\[
\frac{\log \mu_1(n)}{\log n} \rightarrow \frac{1}{2\eta} -1  \ \text{in probability as } n \rightarrow \infty.
\]
If $\eta = 1/2$, and we have the additional condition that $\bar{\delta}(n) - \delta_1(n) = \eta + O(1/\log n)$, then 
\[
\frac{\log \mu_1(n)}{\log\log n} \rightarrow -1 \ \text{in probability as } n \rightarrow \infty.
\]
\end{thm}

Note that Theorem \ref{mainthmdrift} describes a different kind of limit than the other possibility of starting with a countably infinite collection of rank-dependent Brownian motions and describing their stationary behavior. To see the difference compare with the analysis of the Infinite Atlas model done in \cite{palpitman}.

The paper is organized as follows. Section \ref{bfkmodels} discusses the implications of the convergence to Poisson-Dirichlet law for the Banner-Fernholz-Karatzas models of equity markets. In particular, we compute the asymptotic laws of certain functionals of the market weights by resorting to similar computations for the Poisson-Dirichlet law. They include the expected \emph{market entropy} $S=-\sum_i \mu_i\log \mu_i$ and the expected $p$th moment $D_p= \sum_{i}\mu_i^p$. Both these functions are considered in stochastic portfolio theory as measures of \emph{market diversity} (see \cite[page 31]{fern02}). In Section \ref{pp}, we introduce the Poisson-Dirichlet law and discuss its relation with certain Poisson processes on the real line. Section \ref{triangle} contains proofs of all the major theorems in this paper. Finally, in Section \ref{examples}, we discuss counterexamples demonstrating the tightness of the conditions in our theorems.

\section{Application in BFK and related models}\label{bfkmodels}

\smallskip

Fernholz in his 2002 book \cite{fern02} introduces solutions of \eqref{ranksde} to model the time dynamics of the logarithm of the market capitalizations of different companies in an equity market. In other words, he considers a stock market with $n$ companies whose total worth in stocks are coordinate-wise Exponentials of the solution of equation \eqref{ranksde}. A major objective of his work is to explain the following curious empirical fact regarding what are known as the \emph{market weights}. The market weight of a company is defined to be the ratio of its market capitalization to the total market capitalization of all the companies, that is, the proportion of total market capital that belongs to a particular company. The market weight is a measure of the influence that the company exerts on the entire market and have been studied extensively in the economics literature over the decades for their peculiar structure and very stable behavior in time. For example see articles by Hashemi \cite{hashemi00}, Ijiri and Simon \cite{ijirisimon74}, Jovanovic \cite{jovanovic82}, and Simon and Bonini \cite{simonbonini58}. The structure in the market weights arises if one arranges the market weights data in decreasing order and plots their logarithms against the logarithm of their ranks. This log-log plot is referred to as the \emph{capital distribution curve}. For example, if we consider Fernholz's model, and denote by $X$ a solution of SDE \eqref{ranksde}, the ordered market weights are given by the transformation
\begin{equation}\label{mkwts}
\mu_i= \frac{e^{ X_{(i)}}}{\sum_{j=1}^{n} e^{ X_{(j)}}},\quad i=1,2,\ldots,n.
\end{equation}
The capital distribution curve will then be a plot of $\log \mu_i$ versus $\log i$. Empirically this plot exhibits nearly polynomial decay with decreasing ranks. See Figure 1 below (reproduced from \cite{fern02}) which shows capital distribution curves between 1929 and 1999 for all the major US stock markets (NYSE, AMEX, and NASDAQ) combined.    
\begin{figure}[t]
\includegraphics[width=4 in]{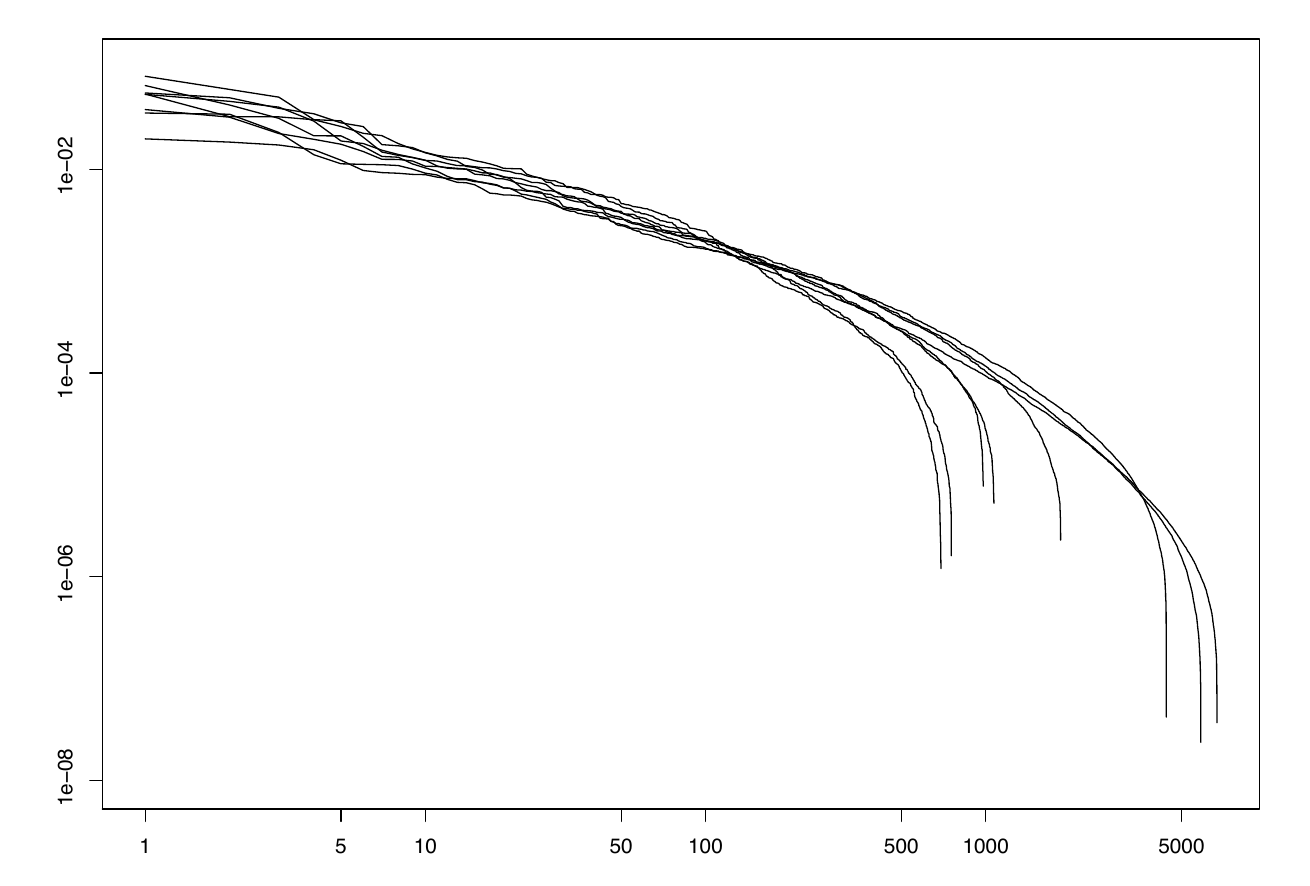}
\caption{Capital distribution curves: 1929-1999}
\end{figure}
To explain this fact, Fernholz considers a special case of \eqref{ranksde} called the \emph{Atlas model} where $\delta_n >0$, and $\delta_i=0$ for all $i\neq n$. He conjectured that the capital distribution curve obtained from the diffusion satisfying the SDE for the Atlas model converges to a stationary distribution under which it is roughly a straight line. 

It is interesting to note that the phase transition phenomenon described in Theorem \ref{mainthmdrift} predicts some remarkable outcomes about models of stock market capitalizations as given in Banner, Fernholz, and Karatzas \cite{atlasmodel}, \cite{fern02}, and \cite{karsurv}. 

As pointed out in \cite{atlasmodel}, it is empirically observed that the companies with the few largest market weights get the least drift, and most movement seems to come from the bottom of the pile. This validates both our assumptions about the drift sequences made in \eqref{conddel2} and \eqref{conddel}. Hence our results show the following universality property of the BFK models. When the number of companies in the market is {large}, one of the following must happen, either, the market share of each company is negligible with respect to the entire market, or the emergence of a {\lq}superpower{\rq} which dominates the market, or a delicate equilibrium under which the market weights are arranged as a particular realization of a Poisson-Dirichlet process. These conclusions hold, no matter what specifically the values of the drifts are, as long as conditions \eqref{conddel2} and \eqref{conddel} are satisfied. The atoms of Poisson-Dirichlet are well-known to decay polynomially (the slope of the log-log plot for PD$(\alpha)$ being roughly $-1/\alpha$) which validates Fernholz's conjecture. 

The Atlas model (\cite{atlasmodel}) is an example where such a phase transition takes place. 
\smallskip

\noindent\textbf{Example.}
Consider the Atlas model with $n$ particles (\cite{atlasmodel}, \cite{palpitman}): for $i=1,2,\ldots,n$, let
\[
d X_i(t) = \eta_n 1\left\{ X_i(t)= X_{(n)}(t)\right\}dt + dW_i(t),\quad \eta_n >0,
\]
where $W$ is an $n$-dimensional Brownian motion. The drift sequence is clearly given by $\delta_n(n)=\eta_n$ and $\delta_i(n)=0$ for $i=1,2,\ldots,n-1$. Thus, $\bar{\delta}(n)= \eta_n/n$, and the drift sequence satisfies condition \eqref{conal}:  
\[
\alpha_k=\sum_{i=1}^k \left( \bar{\delta}(n) - \delta_i(n)\right) = \frac{k\eta_n}{n} > 0,\quad k=1,2,\ldots,n-1.
\]
Thus, by Theorem \ref{theoremN}, there is a unique stationary distribution under which the spacings $Y_{(i)}(t)=X_{(i)}(t)-X_{(i+1)}(t)$ are independent Exponential random variables with rates $2\alpha_i$. Also, note that if $\eta=\lim_n \eta_n/n$ exists in $[0,\infty]$, then
\[
\limsup_{n\rightarrow \infty}\max_i \left(\bar{\delta}(n) - \delta_i(n)\right)=\eta, \quad \lim_{n\rightarrow\infty}\left(\bar{\delta}(n) - \delta_i(n) \right)=\eta, \quad i=1,2,\ldots
\]
Let $Q_n$ be the stationary law of the spacing system in an $n$-particle Atlas model. It follows from Theorem \ref{mainthmdrift} that  as $n$ tends to infinity, the point process of market weights under $Q_n$ either converges to a non-degenrate PD distribution, or to one of the two degenerate limit depending on the value of $\eta$. Since the atoms of the Poisson-Dirichlet laws have polynomial decay, the Atlas model fits into the observed data in Figure 1.

Moreover, the drift sequence trivially satisfies the Lipschitz condition \eqref{totlips} in Theorem \ref{phasegrav} for $\gamma=1/2$. Hence, when $\eta > 1/2$ the market weights go to zero at a polynomial rate. Finally, suppose
\[
\bar{\delta}(n) - \delta_1(n) = \eta_n/n = \frac{1}{2} + O(1/\log n).
\]
Then, the market weights go to zero at a logarithmic rate. In particular, this happens if $\eta_n = n/2$.

\medskip

We introduce another model, the \emph{one-dimensional gravity model}, which satisfies such a transition. 
\smallskip

\noindent\textbf{Example.} For a parameter sequence $\eta_n >0$, let
\begin{equation}\label{whatisgravity}
\quad d X_i(t) = \frac{\eta_n}{n}\sum_{j\neq i} \text{sign}\left(X_j(t)-X_i(t)\right)dt + dW_i(t),\quad i=1,2,\ldots,n.
\end{equation}
Here $(W_1,W_2,\ldots,W_n)$ is an $n$-dimensional Brownian motion. It is straightforward to verify that SDE \eqref{whatisgravity} is a particular case of \eqref{ranksde} where the drift that the $i$th ranked particle gets is $\delta_i=\eta_n(2i-n-1)/n$. The more striking property of this model is the mutual attraction between particles which is due to the one-dimensional gravitational force kernel sign($\cdot$). One can think of the parameter $\eta_n$ as the strength of gravitational pull between the particles. Note that the average drift for the gravity model is zero by symmetry. Moreover, the drift sequence satisfies condition \eqref{conal}:
\[
\alpha_k := \sum_{i=1}^k (\bar{\delta}(n)- \delta_i(n))= {\eta_n}\left(\frac{k(k+1)}{n} - \frac{k(n+1)}{n} \right)= \frac{\eta_n k(n-k)}{n},
\]
which is positive for all $k=1,2,\ldots,n-1$. Thus, from Theorem \ref{theoremN}, it follows that there is a unique stationary measure for the spacings $(Y_1(t), \ldots, Y_{n-1}(t))$ under which the coordinates are independent and the $i$th spacing is distributed as Exponential with rate $2\alpha_i$.

Now, suppose that $\eta:=\lim_n \eta_n$ exists in $[0,\infty]$. Then note that 
\[
\begin{split}
\lim_{n\rightarrow\infty}&\left( \bar{\delta}(n) - \delta_i(n)\right)=\eta, \quad i=1,2,\ldots\\
\lim_n \max_{1\le i\le n}& \left(\bar{\delta}(n) - \delta_i(n)\right)=\lim_n \eta_n = \eta.
\end{split}
\]
Thus, the gravity model satisfies conditions \eqref{conddel2}, \eqref{conddel} for Theorem \ref{mainthmdrift}. Additionally, it follows that
\[ 
\abs{\delta_i(n) - \delta_1(n)}\le \frac{C(i-1)}{n},\quad\forall\; n\in \mathbb{N},
\]
for some constant $C$ depending on $\eta$. Thus, the gravity model also satisfies condition \eqref{totlips} for Theorem \ref{phasegrav}. We can thus conclude that all the conclusions of phase transition in Theorems \ref{mainthmdrift} and \ref{phasegrav} hold for the gravity model as $\eta$ varies from zero to infinity.

\subsection{Some benefits of the Poisson-Dirichlet law.}

The limiting Poisson-Dirichlet law helps us to compute asymptotic laws of functions of the market weights under stationary distribution. Various distributional properties of PD($\alpha$), $0<\alpha<1$, have been studied extensively, and we only use the most important ones for our purpose here. Many more identities together with their various applications can be found in \cite{poidir97}.  

The following theorem, whose proof can be found in Section \ref{triangle}, is a typical example of how such properties can be put to use in the context of BFK models. 

\begin{thm}\label{benefits}
For every $n\in\mathbb{N}$, consider the solution of the SDE \eqref{ranksde} with a drift sequence $\{ \delta_1(n),\delta_2(n),\ldots,\delta_n(n)\}$. Assume that the array of drifts satisfies conditions \eqref{conal} of Theorem \ref{theoremN}, and \eqref{conddel2} and \eqref{conddel} of Theorem \ref{mainthmdrift} with $\eta\in(0,1/2)$. We have the following conclusions. 
\smallskip

\noindent(i) The asymptotic moments of the maximum market weight $\mu_1(n)$ under the stationary law is given by
\begin{equation}\label{moment}
\lim_{n\rightarrow \infty} \e(\mu_1^p(n)) = \frac{1}{\Gamma(p)}\int_0^{\infty}t^{p-1}e^{-t}\psi_{2\eta}^{-1}(t)dt, \quad p > 0,
\end{equation}
where $\psi_{2\eta}(t)=1 + 2\eta\int_0^{1}(1-e^{-tx})x^{-2\eta-1}dx$. In particular, this determines the asymptotic law of the highest market weight.
\smallskip

\noindent(ii) For $p > 2\eta$, let $D_p(n)= \sum_{i=1}^n \mu_i^p(n)$. Then the asymptotic expectation of $D_p(n)$ is given by
\begin{equation}\label{limdp}
\lim_{n\rightarrow\infty}\e D_p(n)=\frac{\Gamma(p-2\eta)}{\Gamma(p)\Gamma(1-2\eta)}.
\end{equation}
\smallskip

\noindent(iii) Define the market entropy $S(n)=-\sum_{i=1}^n \mu_i(n)\log \mu_i(n)$. Then
\begin{equation}\label{entropy}
\lim_{n\rightarrow \infty} \e S(n) = 2\eta\sum_{k=1}^{\infty}\frac{1}{k(k-2\eta)}.
\end{equation}
\end{thm}

\section{Point processes and the Poisson-Dirichlet law}\label{pp}

In this section we introduce the Poisson-Dirichlet law of random point processes and discuss the notion of weak convergence on the underlying space.   

Recall that a point process can be viewed as a random counting measure. For example, the Poisson point process with a $\sigma$-finite intensity measure $\mu$, PPP($\mu$), is defined to be a random point measure with the following properties. For any collection of disjoint Borel subsets $\{E_1,E_2,\ldots,E_k\}$ of the real line satisfying $\mu(E_i)<\infty$ for all $i$, let $N(E_i)$ be the number of elements of the PPP($\mu$) that are contained in $E_i$. Then the coordinates of the random vector $(N(E_1),N(E_2),\ldots,N(E_k))$ are independent and the $i$th coordinate is distributed as Poisson($\mu(E_i)$). If the random point process has a maximum atom almost surely, it is customary to order the atoms as a random decreasing sequence of real numbers. This happens, for example, when the intensity measure $\mu$ for PPP($\mu$) is such that $\mu[1,\infty)<\infty$.

To discuss weak convergence of such random sequences it is necessary to define an underlying Polish space. The separable complete metric space $(\mcal{S},\mathbf{d})$ we use for this article is the space of all real decreasing sequences with the metric of pointwise convergence, i.e.,
\begin{equation}\label{sd}
\mcal{S}=\{x=(x_i),\; x_1 \ge x_2\ge \ldots  \},\quad {\bf d}(x,y)=\sum_{i=1}^{\infty}\frac{\abs{x_i-y_i}\wedge 1}{2^i}. 
\end{equation}
The fact that this is a Polish space and other properties can be found in page 45 of~\cite{resnick}. From now on, our point processes will be random sequences taking values in the state space $(\mcal{S},\mathbf{d})$. Any finite sequences $(x_1,x_2,\ldots,x_n)$ will correspond to an infinite sequence $(x_1,x_2,\ldots,x_n,-\infty,-\infty,-\infty,\ldots)$ in $(\mcal{S},\mathbf{d})$.

Here is a typical example of a point process which will be of use to us later on. Consider a sequence of independent standard Exponential random variables $\eta_i,\;i=1,2,\ldots$. Let $\eta_1(n) \ge \eta_2(n) \ge \ldots\ge \eta_n(n)$ be the decreasingly ordered rearrangement among the first $n$ values of the sequence. That is 
\[
\eta_1(n)= \max_{1\le i\le n}\eta_i,\quad \ldots,\quad \eta_n(n)=\min_{1\le i\le n}\eta_i.
\]

\begin{lem}\label{orderedexp}
The random point process $\pp_n^*=(\eta_i(n)-\log n),\;i=1,2,\ldots,n$, converges weakly to a Poisson point process on the line with intensity $d\mu^* =e^{-x}dx$.
\end{lem}

\begin{proof}
Convergence of the extreme order statistics of iid data to a Poisson point process is a standard exercise in extreme value statistics. For example, see the book by Resnick \cite{resnick}. In general, it follows from the fact that suppose, for every $n\ge 1$, $X_{1}, X_{2}, \ldots, X_{n}$ are $n$ iid data with law $P_n$.  Further assume that there is a measure $\mu$ such that $\lim_{n\rightarrow \infty}nP_n(A)=\mu(A)$ for all Borel subsets $A$ of $\rr$. Then the random point measure induced by the $n$ data points converge weakly to a Poisson Random measure with intensity $\mu$. When the iid data comes from Exponential$(1)$ distribution shifted by $\log n$, $\mu(A)=\int_A e^{-x}\dx x=\mu^*(A)$, and this proves the Lemma.   
\end{proof}

The other point process important in this paper is a one parameter family of laws of point processes known as the Poisson-Dirichlet with parameter $\alpha$, PD($\alpha$), for $0 < \alpha < 1$. The PD family is fundamental and comes up in several contexts. It can also be seen as a part of a wider two parameter family of laws. For an excellent survey of old and new results we suggest the article by Pitman and Yor \cite{poidir97}. 

PD($\alpha$) can be defined in terms of a PPP with intensity measure $d\mu_{\alpha}=x^{-\alpha - 1}dx$ on $(0,\infty)$. Let $(X_i)$ be the decreasingly arranged atoms of PPP($\mu_{\alpha}$). Then, by Proposition 10 in \cite{poidir97}, the point process
\[
V_i= \frac{X_i}{\sum_{j=1}^{\infty} X_j}, \quad i=1,2,\ldots,
\]
is distributed as PD($\alpha$). We take this to be the definition of the PD family. It can be shown that each $V_i$ is strictly between zero and one, and $\sum_i V_i =1$. For more details, see \cite{poidir97}.

Now the Poisson point process in Lemma \ref{orderedexp} has a very fundamental connection with the Poisson-Dirichlet (PD) process which is made clear through the next lemma.

\begin{lem}\label{exppd}
Let $\{Y_i\}$ be a realization of a Poisson point process in decreasing order with intensity $\mu^*$ as in the last lemma. Then for any $\beta > 1$, the law of the point process generated by the sequence
\[
c_i = \frac{\exp \beta Y_i}{ \sum_j \exp \beta Y_j}, \quad i=1,2,\ldots,
\]
is the Poisson-Dirichlet distribution $\text{PD}(1/\beta)$.
\end{lem}

\begin{proof}
From Proposition 5.2 in \cite[page 121]{resnick}, it follows that the point process $\exp \beta Y_i$ is a Poisson point process on the real line of intensity $\beta^{-1}x^{-1/\beta - 1}dx$. Since $\beta > 1$, the result follows from the definition of PD$(1/\beta)$ once we note that the constant $\beta^{-1}$ in the intensity measure is a scaling parameter which gets cancelled when we take the ratio.
\end{proof}

The following theorem has been proved by Talagrand in \cite[Lemma 1.2.1, page 15]{sping} in the case of iid Gaussians. In Section \ref{triangle}, we follow more or less the same argument to prove it for Exponential random variables. The difference in the two arguments is merely due to the fact that the tails of Exponentials are fatter than those of Gaussians.

\begin{thm}\label{iidexppd}
Consider again the ordered Exponential random variables $(\eta_i(n), \; i=1,2,\ldots,n)$. For any $\beta > 1$ define the point process
\[
U^*_n= \left( \frac{\exp\{\beta \eta_i(n)\}}{\sum_{j=1}^n \exp\{\beta \eta_j(n)\}},\quad i=1,2,\ldots,n\right).
\]
Then, as $n$ tends to infinity, $U^*_n$ converges weakly to the PD($1/\beta$) law.
\end{thm}

The difficulty in proving Theorem \ref{iidexppd} lies in the fact that it does not follow from the convergence of the numerators of $U^*_n$ to the Poisson point process of Lemma \ref{orderedexp}. Although Poisson-Dirichlet can be obtained from the limiting sequence as described by Lemma \ref{exppd}, the convergence of the point processes of the extreme statistics is in the sense of distributions which does not guarantee convergence of its sum or the ratio. The hard work lies in estimating the denominator so that one can show the convergence of the sequence of ratios directly to Poisson-Dirichlet. The same problem will be faced later, greatly amplified, when we prove Theorem \ref{mainthmdrift}.

Weak convergence on the space $(\mcal{S},\mathbf{d})$ can sometimes be conveniently proved by suitable coupling thanks to the following version of Slutsky's theorem.

\begin{lem}\label{ppcoupling}
Suppose on the same probability space we have two sequences of point processes $X(n)=\{ X_i(n),\;i=1,2,\ldots\}$ and $X'(n)=\{ X_i'(n),\;i=1,2,\ldots\}$, $n\in \mathbb{N}$. Suppose $X(n)$ converges in law to $X$, and 
$X_i(n)-X'_i(n)$ goes to zero in probability for every $i$, as $n$ tends to infinity. Then $X'(n)$ also converges weakly to $X$.
\end{lem}

\begin{proof}
The result follows from Slutsky's Theorem (Theorem 3.4 in \cite{resnick}) once we show that ${\bf d}(X(n),X'(n)) \stackrel{P}{\rightarrow} 0$. To achieve this, note that the random variables $\abs{X_i(n) - X_i'(n)}\wedge 1 \le 1$ and converges to zero in probability by our hypothesis. A simple application of dominated convergence theorem imply that
\[
\lim_{n\rightarrow \infty}\e \;{\bf d}(X(n),X'(n)) = \lim_{n\rightarrow \infty}\e \sum_{i=1}^{\infty} \frac{\abs{X_i(n)-X_i'(n)}\wedge 1}{2^i} =0.
\]
This shows that ${\bf d}(X(n),X'(n))$ converges to zero in probability and proves the lemma.
\end{proof}

Now, the support of the PD law is over sequences which have a special structure, namely, all the elements are non-negative and they add up to one. It will be easier for our arguments if we now restrict our sequence space to exactly such sequence with an apparently stronger metric. Let $\mcal{S}'$ be the set of all sequences $(x_i)$ such that $x_1 \ge x_2 \ge \ldots \ge 0$ for all $i$ and $\sum_i x_i \le 1$. For any two sequences $(x_i)$ and $(y_i)$ we also consider the $\lone$-norm between them
\begin{equation}\label{whatisdp}
{\bf d}'(x,y)= \sum_{i=1}^{\infty}\abs{x_i-y_i}.
\end{equation}
Weak convergence of measures on $(\mcal{S}',{\bf d}')$ is determined by functions which are continuous with respect to the $\lone$-norm. In fact, we will, without loss of generality, consider functions which are Lipschitz with Lipschitz coefficient not more than one. That is to say, a sequence of probability measures $\mu_n$ on $(\mcal{S}',{\bf d}')$ converges weakly to another probability measure $\mu$ on the same space if and only if
\begin{equation}\label{wksprime}
\int f d\mu_n \rightarrow \int f d \mu,
\end{equation}
for all $f$ which are Lipschitz with Lipschitz coefficients not more than one. 

Finally a note on the notation used in the following section. The Poisson-Dirchlet law is a probability law on $(\mcal{S}',{\bf d}')$. For any test function $f$, we will denote its expectation with respect to any PD law by PD($f$). The parameter of the PD law will be obvious from the context.

\section{Proofs}\label{triangle}

\subsection{Proof of Theorem \ref{iidexppd}}

Let us recall the set-up: $\eta_1,\eta_2,\ldots$ is a sequence of iid Exponential one random variables, and $\eta_1(n) \ge \eta_2(n) \ge \ldots \ge \eta_n(n)$ be the ordered values of the first $n$ random variables. For any $\beta > 1$, we want to find the limiting law of the point process 
\[
U_n^*= \left( \frac{\exp \left\{\beta \eta_i(n)\right\}}{\sum_{j=1}^n \exp \left\{\beta \eta_j(n)\right\}},\quad i=1,2,\ldots,n \right). 
\]

The first step is to fix a real number $b$, and define
\begin{equation}
\begin{split}
h_i(n):=\eta_i(n) - \log n,&\qquad u_i(n):= e^{\beta h_i(n)},\\
S(n):= \sum_{i=1}^n u_i(n),& \qquad S_b(n):= \sum_{i=1}^n u_i(n)1\{h_i(n) \ge b \}.
\end{split}
\end{equation}
Let us denote $w_i(n)=u_i(n)/S(n)$, $i=1,2,\ldots,n$, which are the coordinates of the point process $U_n^*$. We also define another point process depending on $b$:
\begin{equation}\label{cutcut}
w_i(b,n) = \begin{cases}
0,&\quad \text{if}\; h_i(n) < b,\\
u_i(n)/S_b(n),&\quad \text{if}\; h_i(n) \ge b.
\end{cases}
\end{equation}
The following lemma, exactly as in \cite[Lemma 1.2.4, pg 17]{sping}, estimates the difference between the two point processes.

\begin{lem}
\[
\sum_{i=1}^n \abs{w_i(n) - w_i(b,n)} = 2\frac{S(n)-S_b(n)}{S(n)}.
\]
\end{lem}

\begin{lem}
Given $\epsilon > 0$, we can find $b \in \rr$ such that
\begin{equation}\label{somebound}
\limsup_{n \rightarrow \infty}P\left(\frac{S(n) - S_b(n)}{S(n)} \ge \epsilon \right) \le \epsilon.
\end{equation}
\end{lem}

\begin{proof} Note that if $S(n) \le \exp \beta x$ for some $x\in \rr$, then $\eta_i -\log n \le x$ for all $i=1,2,\ldots,n$. The probability of this event is $\left( 1 - e^{-x}/n \right)^n$, which tends to $\exp(-\exp (-x))$ as $n$ tends to infinity. This proves that for any $\epsilon$, there is a $\delta > 0$, such that for all large enough $n$, one has $P(S(n) \le \delta) \le  \epsilon / 2$.

We will show that given $\epsilon, \delta$ as above, there is a $b$ such that for all $n$ large enough
\begin{equation}\label{cauchy}
P\left( S(n) - S_b(n) \ge \epsilon\delta \right) \le \frac{\epsilon}{2}.
\end{equation}
It follows for that choice of $b$ that for all $n$ large enough we have
\[
P\left(\frac{S(n) - S_b(n)}{S(n)} \ge \epsilon \right) \le \epsilon.
\]
In other words, \eqref{somebound} holds.

To show \eqref{cauchy}, note that by symmetry
\[
\begin{split}
\e(S(n) - S_b(n)) &= n \e\left(e^{\beta (\eta_1-\log n)}1\{\eta_1-\log n \le b \} \right)\\
&= n^{1-\beta}\int_0^{\log n+b} e^{\beta x}e^{-x}dx\\
&= \frac{1}{n^{\beta-1}(\beta-1)}\left[ n^{\beta -1}\exp\left( (\beta-1)b  \right) - 1 \right]\\
&= \frac{e^{b(\beta-1)}}{\beta -1} - \frac{1}{n^{\beta-1}(\beta-1)}.
\end{split}
\]
Since $\beta > 1$, we can choose $b$ to be a large negative number such that for all $n$ sufficiently large $\e(S(n) - S_b(n)) \le \epsilon^2\delta/2$. Now \eqref{cauchy} follows by Markov's inequality. 
\end{proof}

\begin{proof}[Proof of Theorem \ref{iidexppd}] The proof now follows exactly as in Talagrand's proof of Theorem 1.2.1 in \cite[page 19]{sping}. The basic idea being that from the last lemma it follows that the denominator in $U_n^*$ can be approximated by only about a Poisson many variables no matter what $n$ is. The rest follows by pointwise convergence of these finitely many coordinates to the corresponding atoms of the Poisson point process which in turn approximates the Poisson-Dirichlet ratios. We skip the details.
\end{proof}

\subsection{Proof of Theorem \ref{mainthmdrift}}
The proof of Theorem \ref{mainthmdrift} comes in several steps. The main idea is the following: compare the ordered process in stationarity with the point process of the order statistics of Exponential distribution as considered in Theorem \ref{iidexppd}. Condition \eqref{conddel2} can be interpreted as saying that the top few atoms of both point processes have the same joint distribution up to a shift. The PD convergence requires exponentiation and taking a ratio as in Theorem \ref{iidexppd}. With the help of condition \eqref{conddel} we will show that when we exponentiate, the top few atoms gets so large compared to the rest that the denominator of the ratio in \eqref{pprev} is basically a finite sum. This is the main technical difficulty in the proof which requires computations of suitable estimates. The PD convergence then follows from Theorem \ref{iidexppd}. 

The first step is the following lemma.

\begin{lem}\label{mainthm}
Suppose a triangular array of constants $(\beta_i(n),\; 1\le i\le n-1)$, $n\ge 2$, satisfies the following two conditions 
\begin{equation}\label{limiti}
\limsup_{n\rightarrow \infty}\left[\max_{1\le i\le n-1}\frac{\beta_{i}(n)}{i}\right] \le 1,\quad \text{and}\quad \lim_{n\rightarrow \infty} \beta_{i}(n) = i,\quad i=1,2,3,\ldots
\end{equation} 
Corresponding to this array of constants, let us take an array of Exponential random variables $(\xi_1(n), \xi_2(n),\ldots,\xi_{n-1}(n))$, $n\in \mathbb{N}$, all independent with $\e(\xi_i(n))=1/\beta_i(n)$. For each $n\in \mathbb{N}$ and $\beta > 0$, let $U_n(\beta)$ be the point process with $n$ coordinates given by
\[
\begin{split}
U_n(\beta)(j)&=\frac{\exp \left(\beta \sum_{i=j}^{n-1} \xi_i(n)\right)}{1 + \sum_{k=1}^{n-1} \exp \left(\beta \sum_{i=k}^{n-1} \xi_i(n)\right)  },\quad j=1,2,\ldots,n-1,\quad \text{and}\\
U_n(\beta)(n)&=\frac{1}{1 + \sum_{k=1}^{n-1} \exp \left(\beta \sum_{i=k}^{n-1} \xi_i(n)\right)}.
\end{split}
\]
Then
\begin{itemize}
\item[(i)] For any $\beta > 1$, the sequence $U_n(\beta)$ converges in distribution to the Poisson-Dirchlet law PD$(1/\beta)$ as $n \rightarrow \infty$.
\item[(ii)] For any positive sequence $\{\delta_n\}$ such that $\delta_n \le 1$ and $\lim_{n\rightarrow \infty}\delta_n=\beta \in [0,1]$, the point process $U_n(\delta_n)$ converges in law to unit mass on the sequence of all zeroes. In particular, $U_n(\beta)$ converges to the zero sequence when all the $\delta_n$'s are identically equal to $\beta$ for some $\beta \le 1$.
\end{itemize}
\end{lem}
Let us first complete the proof of Theorem \ref{mainthmdrift} using the above result. First suppose that $\eta > 0$ in \eqref{conddel2} and \eqref{conddel}. In the notation used in Lemma \ref{mainthm}, we define
\begin{equation}\label{betainthm}
\quad\beta_i(n)= \frac{1}{\eta}\sum_{j=1}^i\left( \bar{\delta}(n)-\delta_j(n)\right), \ \ i =1,\ldots,n-1.
\end{equation}
Note that each $\beta_i(n) > 0$ by \eqref{conal}.
Thus, by assumptions \eqref{conddel2} and \eqref{conddel}, we get
\begin{equation}\label{basicineq}
\begin{split}
\limsup_n\max_{1\le i\le n-1}\frac{\beta_{i}(n)}{i}&=\limsup_n\left(\max_{1\le i\le n-1}\frac{1}{i}\sum_{j=1}^i\frac{ \bar{\delta}(n)-{\delta_j}(n)}{ \eta}\right)\le 1.
\end{split}
\end{equation}
Also, clearly
\[
\lim_{n\rightarrow \infty}{\beta_i(n)}= \lim_{n\rightarrow \infty}\sum_{j=1}^i\frac{ \bar{\delta}(n)-{\delta_j}(n)}{\eta}=i.
\]
Thus the array $\{\beta_i(n)\}$ satisfies both the conditions in \eqref{limiti}. Now, by Theorem~\ref{theoremN}, for the model with $n$ particles, the joint law of the $n$ spacings $(X_{(i)}(t)-X_{(i+1)}(t),\;i=1,2,\ldots,n-1)$ under the stationary distribution is the same as that of a vector of independent Exponentials $(Y_1(n), Y_2(n),\ldots,Y_{n-1}(n))$ such that $Y_i(n)$ is Exponential with rate 
\[
2\sum_{j=1}^i(\bar{\delta}(n)-\delta_j(n))=2\eta\beta_i(n),
\]
which is the joint law of $(\xi_i(n)/2\eta, \ldots, i=1,2,\ldots,n-1)$.
Thus, the decreasingly arranged market weights in the $n$ particle system are in law given by
\begin{align*}
\mu_i(n)&= \frac{\exp(X_{(i)}(n)-X_{(n)}(n))}{\sum_{j=1}^{n} \exp(X_{(j)}(n)-X_{(n)}(n))}\\
&\stackrel{\mcal{L}}{=}  \frac{\exp \left\{\frac{1}{2\eta}\sum_{j=i}^{n-1} \xi_j(n)\right\}}{\sum_{k=1}^{n} \exp \left\{ \frac{1}{2\eta}\sum_{j=k}^{n-1} \xi_j(n)\right\} }, \quad i=1,2,\ldots,n.
\end{align*}

From Lemma \ref{mainthm} it is immediate that if $\eta\in(0,1/2)$, the point process of market weights converges weakly to the Poisson-Dirichlet law PD($2\eta$) as $n$ tends to infinity, and if $\eta \ge 1/2$, the point process converges to zero. This proves conclusions (i) and (ii) in Theorem \ref{mainthmdrift}.  

Let us now prove part (iii) of Theorem \ref{mainthmdrift}. Assume that $\eta=0$. For the $n$ particle system, the largest market weight, $\mu_1(n)$, is given by
\[
\frac{e^{X_{(1)}(n)}}{\sum_{i=1}^n e^{X_{(i)}(n)}}= \frac{1}{1+\sum_{i=1}^{n-1} \exp\{-\sum_{j=1}^i Y_{j}(n)\}},
\]
where $Y_j(n)$ is Exponential with rate $2\alpha_i(n)$, where $\alpha_i(n)=\sum_{j=1}^i(\bar{\delta}(n)-\delta_j(n))$. It is obvious that $\mu_1(n)\le 1$. We will show that  $\mu_1(n)^{-1} \rightarrow 1$  in probability, which implies that $\mu_1(n) \rightarrow 1$ in probability too. Since all the market weights are nonnegative and add up to one, this forces the rest of them to converge to zero, and our theorem is proved.

Now, recall that if $Y\sim\text{Exp}(\lambda)$, then $\e e^{-Y}$ is $(1+1/\lambda)^{-1}$. Thus
\begin{equation}\label{somebnd}
\begin{split}
1\le \e\left(1/\mu_1(n)\right)&= 1 + \sum_{i=1}^{n-1}\prod_{j=1}^{i}\e\left(e^{-Y_{j}(n)} \right)\\
&= 1 + \sum_{i=1}^{n-1}\prod_{j=1}^{i}\frac{1}{1+1/2\alpha_{j}(n)}.
\end{split}
\end{equation}
By assumptions \eqref{conddel2} and \eqref{conddel}, there exists $n_0\in\mathbb{N}$ such that
\[
\max_{1\le i \le n}(\bar{\delta}(n)-\delta_i(n)) \le 2(\bar{\delta}(n) - \delta_1(n)),\quad \forall\; n\ge n_0.
\]
If we denote $c_n=\bar{\delta}(n) - \delta_1(n)$, it follows from the definition of $\alpha_j(n)$ that for all $n\ge n_0$ we have $\alpha_j(n)\le 2j c_n$, and hence
\[
\sum_{i=1}^{n-1}\prod_{j=1}^{i}\frac{1}{1+1/2\alpha_j(n)}\le \sum_{i=1}^{n-1}\prod_{j=1}^{i}\frac{1}{1+1/4jc_n}.
\]
Let 
\[
q_i(n)= \begin{cases}
\prod_{j=1}^{i}\left(1+1/4jc_n\right)^{-1},&\quad i\le n,\\
0,&\quad \text{otherwise}.
\end{cases}
\]
Without loss of generality, we can assume, by \eqref{conddel2}, that $c_n \le 1/8$ for all $n \ge n_0$. Thus
\[
q_i(n) \le \prod_{j=1}^{i}\left(1+2/j\right)^{-1}\le \prod_{j=1}^i \frac{j}{j+2}= 2\frac{ i!}{(i+2)!}\le Ci^{-2},
\]
where $C$ is some positive constant. Thus, $q_i(n)$ is dominated by a summable series for all $n \ge n_0$. Also note that, since $c_n \rightarrow 0$ by \eqref{conddel2}, we get $\lim_{n\rightarrow \infty}q_i(n)=0,\; \forall\; i$. Hence we can apply the dominated convergence theorem to get
\[
\limsup_{n\rightarrow\infty}\left\lvert\e\left(1/\mu_1(n)\right)-1\right\rvert \le \lim_{n\rightarrow \infty} \sum_{i=1}^{\infty}q_i(n) =0.
\]
Thus, we have shown that $\mu_1^{-1}(n)$ converges to one in probability, and this completes the proof of Theorem \ref{mainthmdrift}.
\vskip.2in

Now let us begin the proof of Lemma \ref{mainthm}. On a possibly enlarged sample space, consider an independent sequence $\{\eta_1,\eta_2,\ldots\}$ of iid Exponential random variables with mean one. Let $\eta_1(n) \ge \eta_2(n) \ge \ldots \ge \eta_{n}(n)$ denote the first $n$ random variables $\{\eta_1,\ldots,\eta_{n}\}$ arranged in decreasing order. As considered in Lemma \ref{orderedexp}, let $\pp^*_n$ denote the point process given by 
\begin{equation}\label{ppnstar}
\pp^*_n= \left(\eta_i(n)-\log n,\; i=1,2,\ldots,n\right), \quad n\in \mathbb{N}.
\end{equation}
For any $b \in \mathbb{N}$, define a new sequence of points by
\begin{equation}\label{whatisy}
Y_j'(n) = \begin{cases}
\eta_j(n),&\quad j > b\\
\eta_{b+1}(n) + \sum_{i=j}^{b} \xi_i(n),&\quad 1\le j \le b. 
\end{cases} 
\end{equation}
Let $\tilde{\pp}_n(b)$ be the new point process $\tilde{\pp}_n(b) = \left(Y_i'(n)-\log n\right),\;i=1,2,\ldots$. We prove below that, for every $b\in \mathbb{N}$, as $n$ tends to infinity, the law of $\tilde{\pp}_n(b)$ converges to the same limiting law as that of $\pp^*_n$ as described in Lemma \ref{orderedexp}.

\begin{lem}\label{convergeetilde}
For every fixed $b \in\mathbb{N}$, as $n$ tends to infinity, the random point process $\tilde{\pp}_n(b)$ converges to a Poisson point process on the line with intensity $d\mu^*=e^{-x}dx$. 

Moreover, for every $\beta > 1$, the normalized Exponential point process
\begin{equation}\label{whatisutilde}
\widetilde{U}_n(b)=\left\{\frac{\exp\{\beta Y_i(n)\}}{\sum_{j=1}^n \exp\{\beta Y_j(n)\}},\quad i=1,2,\ldots,n\right\}
\end{equation}
converges in law to the Poisson-Dirichlet distribution with parameter $1/\beta$.
\end{lem}

\begin{proof}
Our main tool will be to apply Lemma \ref{ppcoupling}. Let us recall the \emph{R\'enyi representation} (\cite[page 114]{resnick}): 
\begin{equation}\label{renyi}
\eta_j(n) \stackrel{\mcal{L}}{=} Y_j^*(n):=\sum_{i=j}^{n} {\eta_i}/{i}, \quad 1\le j\le n,\quad n\in \mathbb{N},
\end{equation}
where $\stackrel{\mcal{L}}{=}$ denotes equality in law. Also, trivially, the joint distribution of the sequence $\{\xi_i(n),\;1\le i \le n-1\}$ is the same as that of $\{\eta_i/\beta_i(n),\;1\le i\le n-1 \}$.

Thus, analogous to \eqref{whatisy}, it is clear that we can suitably redefine the sequence $\{\eta_i\}$ so that
\begin{equation}\label{whatisxprime}
Y_j'(n)= \begin{cases}
Y_j^*(n)=\sum_{i=j}^{n} {\eta_i}/{i},&\quad j > b\\
\sum_{i=b+1}^{n} {\eta_i}/{i} + \sum_{i=j}^{b} {\eta_i}/{\beta_i(n)},&\quad 1\le j \le b.
\end{cases}
\end{equation}
For the rest of this argument we refer to $\pp^*_n$ as the point process given by $(Y_i^*(n)-\log n, \quad 1\le i\le n)$. As before, $\tilde{\pp}_n$ is the sequence of points $(Y'_i(n)-\log n,\;1\le i\le n)$. 

Recall that by Lemma \ref{orderedexp}, $\pp_n^*$ converges in law to a Poisson point process with intensity $d\mu^*=e^{-x}dx$ on the entire real line. To establish the same limit for the point process $\tilde{\pp}_n$ we will verify the assumptions of Lemma \ref{ppcoupling}. 

Now, by our construction, the difference between the $j$th coordinates of the two point processes $\abs{\pp_n^*(j)-\tilde{\pp}_n(j)}$ is equal to
\[
\abs{Y_j^*(n) - Y_j'(n)}\le
\begin{cases}
0,\quad \text{if}\; j > b,\\
\sum_{i=j}^{b} \abs{\frac{1}{\beta_i(n)} - \frac{1}{i}}\eta_i, \quad \text{if}\; 1 \le j \le b,
\end{cases}
\]
And thus
\[
\e\abs{Y_j^*(n) - Y_j'(n)} \le\sum_{i=j}^{b} \abs{\frac{1}{\beta_i(n)} - \frac{1}{i}}\e (\eta_i)= \sum_{i=j}^{b}\abs{\frac{1}{\beta_{i}(n)} - \frac{1}{i}}.
\]
Now, the right hand side of the last equation goes to zero as $n$ tends to infinity by the second condition in \eqref{limiti}. It follows that the $i$th coordinate of $\pp^*_n - \tilde{\pp}_n$ goes to zero in probability, and the first part of the lemma follows by Lemma \ref{ppcoupling}.
\medskip

Now, for any $\beta > 1$, by the continuous mapping theorem it follows that the difference between any coordinate of the point processes $(\exp (\beta Y'_i(n)),\;i\le n)$ and $(\exp(\beta Y_i^*(n)),\; i\le n)$ goes to zero in probability. Also, since the two sequences differ only at finitely many coordinates, it follows that
\[
\abs{\sum_i e^{\beta Y'_i(n)} - \sum_i e^{\beta Y^*_i(n)}} \stackrel{P}{\rightarrow} 0.
\]
By another application of the \emph{continuous mapping theorem} (see, e.g., \cite[Theorem 3.1, page 42]{resnick}) we can conclude that the point process 
\begin{equation}\label{whatisunt}
\widetilde{U}_n(b)=\left\{\frac{\exp\{\beta Y'_i(n)\}}{\sum_{j=1}^n \exp\{\beta Y'_j(n)\}},\quad i=1,2,\ldots,n\right\}
\end{equation}
converges to the same limiting law as the process 
\[
{U}^*_n=\left\{\frac{\exp\{\beta Y^*_i(n)\}}{\sum_{j=1}^n \exp\{\beta Y^*_j(n)\}},\quad i=1,2,\ldots,n\right\}.
\]
By Theorem \ref{iidexppd} and the equality in law \eqref{renyi}, we get that $U^*_n(b)$ converges to a PD($1/\beta$) law as $n$ tends to infinity. This completes the proof of the lemma.
\end{proof}

Define 
\begin{equation}\label{whatisx}
X_j(n) := \begin{cases}
\sum_{i=j}^{n-1} \xi_i(n)&, \quad 1\le j \le n-1,\\
0&, \quad j=n.
\end{cases}
\end{equation}
Next, we will show that as $b$ tends to infinity, the limiting point process for $\widetilde{U}_n(b)$ is the same as the limiting point process (depending on the parameter $\beta$)
\begin{equation}\label{whatisun}
U_{n} =\left\{ \frac{\exp {\beta X_i(n)}}{\sum_{j=1}^n \exp {\beta X_j(n)}},\quad i=1,2, \ldots, n\right\}.
\end{equation}
This will prove our objective that the limiting distribution for the point process $U_n$ is given by Poisson-Dirchlet with parameter $1/\beta$. We start with the following lemma. 

\begin{lem}\label{approxb}
Let $\{\eta_i,\;i=1,2,\ldots\}$ be a sequence of independent  Exponential one random variables. For $\beta > 1$, let us define 
\[
X_i'(n) := \begin{cases}
\sum_{j=i}^{n-1} {\eta_j}/{\beta_j(n)},\quad 1\le i \le n-1,\\
0,\quad i=n.
\end{cases}
\]
and let
\[
U_{n} =\left\{ \frac{\exp {\beta X'_i(n)}}{\sum_{j=1}^n \exp {\beta X'_j(n)}},\quad i=1,2, \ldots, n\right\}. 
\]
Let $\widetilde{U}_n(b)$ be defined as in \eqref{whatisunt}. Then 
\[
\limsup_{n\rightarrow \infty}\e\; {\bf d}'\left(U_{n},\widetilde{U}_n(b)\right)\le \frac{C}{b^{\sqrt{\beta}-1}},
\]
where $C$ is a constant depending on $\beta$ and $\mathbf{d}'$ is the $\lone$ distance defined in \eqref{whatisdp}.
\end{lem}

\begin{proof}
Note that, the sequence $\{X'_i(n),\;i=1,2,\ldots,n\}$ have the same joint law as the original sequence $\{X_i(n),\;i=1,2,\ldots,n\}$ in \eqref{whatisx}. Thus, the new definition of the point process $U_n$ is consistent with the old one in \eqref{whatisun}.

Now, suppressing the dependence of $b$ and $\beta$, let
\[
U_i(n)=  \frac{\exp {\beta X'_i(n)}}{\sum_j \exp {\beta X'_j(n)}},\quad \widetilde{U}_i(n)=\frac{\exp {\beta Y'_i(n)}}{\sum_j \exp {\beta Y'_j(n)}}.
\]
Then, by definition, the difference $\sum_{i=1}^n \abs{U_i(n)-U_i(n)}$ is equal to
\begin{equation}\label{decompsum}
\begin{split}
& \sum_{i=1}^{n} \abs{\frac{\exp {\beta (X'_i(n)-X'_1(n))}}{\sum_j \exp {\beta (X'_j(n)-X'_1(n))}} - \frac{\exp {\beta (Y'_i(n)-Y'_1(n))}}{\sum_j \exp {\beta (Y'_j(n)-Y'_1(n))}} } \\
&= \sum_{i\le b+1} \abs{ \frac{\exp {\beta (X'_i(n)-X'_1(n))}}{\sum_j \exp \left\{\beta (X'_j(n)-X'_1(n))\right\}}-\frac{\exp {\beta (Y'_i(n)-Y'_1(n))}}{\sum_j \exp \left\{\beta (Y'_j(n)-Y'_1(n))\right\} } }\\
&+ \sum_{i > b+1}\abs{\frac{\exp {\beta (X'_i(n)-X'_1(n))}}{\sum_j \exp {\beta (X'_j(n)-X'_1(n))}} - \frac{\exp {\beta (Y'_i(n)-Y'_1(n))}}{\sum_j \exp {\beta (Y'_j(n)-Y'_1(n))}} }. 
\end{split}
\end{equation}
Let $Q_n$, $R_n$, and $S_n$ be defined (all depending on $b$ and $\beta$) as the partial sums
\[
\begin{split}
Q_n &= \sum_{i> b+1} \exp\left\{\beta(X'_i(n) - X'_1(n))\right\},\quad R_n= \sum_{i> b+1} \exp\left\{\beta(Y'_i(n) - Y'_1(n))\right\},\\
S_n&= \sum_{i \le b+1} \exp\left\{\beta(X'_i(n) - X'_1(n))\right\}=\sum_{i \le b+1}\exp \left\{-\beta \sum_{k < i}\eta_k/\beta_k(n)\right\}\\
&=\sum_{i \le b+1} \exp\left\{\beta(Y'_i(n) - Y'_1(n))\right\}.
\end{split}
\]
One can rewrite the first term on the RHS of \eqref{decompsum} as
\[
\begin{split}
\sum_{i\le b+1} &\exp \left\{-\beta \sum_{k < i}\eta_k/\beta_k(n)\right\}\abs{ \frac{1}{Q_n+S_n}-\frac{1}{R_n+S_n } }\\
 &= \abs{\frac{S_n}{Q_n+S_n} - \frac{S_n}{R_n+S_n}}= \abs{\frac{Q_n}{Q_n+S_n} - \frac{R_n}{R_n+S_n}}\\
 &\le \frac{Q_n}{Q_n+S_n} + \frac{R_n}{R_n+S_n}\le Q_n+R_n.
\end{split}
\]
The final inequality follows by noting that $S_n$ is greater than one. The rest of the sum on the RHS of \eqref{decompsum} is bounded above by
\[
\begin{split}
\sum_{i> b+1}\frac{\exp {\beta (X'_i(n)-X'_1(n))}}{\sum_j \exp {\beta (X'_j(n)-X'_1(n))}}
&+ \sum_{i> b+1} \frac{\exp {\beta (Y'_i(n)-Y'_1(n))}}{\sum_j \exp {\beta (Y'_j(n)-Y'_1(n))}}\\
&=\frac{Q_n}{Q_n+S_n} + \frac{R_n}{R_n+S_n}\le Q_n+R_n. 
\end{split}
\]
Hence the expected distance $\e\;{\bf d}'(U(n),\tilde{U}(n))$ is bounded above by $2\e(Q_n + R_n)$. The proof is now completed by applying the next lemma.
\end{proof}

\begin{lem}
\[
\limsup_{n \rightarrow \infty} \e Q_n\le \frac{C_1}{b^{\sqrt{\beta}-1}},\qquad \limsup_{n \rightarrow \infty} \e R_n\le \frac{C_1}{b^{\sqrt{\beta}-1}},
\]
where $C_1$ is a constant depending on $\beta$.
\end{lem}

\begin{proof}
We will first find a bound on the expected value of $Q_n$. Note that 
\[
Q_{n} = \sum_{i> b+1}^{n} \exp\left\{\;-\beta \sum_{k < i} \eta_k/\beta_k(n)\;\right\}.
\]
where $\{\eta_k,\;k=1,2,\ldots n-1\}$ are a sequence of Exponential($1$) and the array $\{\beta_k(n),\;k=1,2,\ldots,n-1\}$ satisfies condition \eqref{limiti}. 
\[
\begin{split}
\e Q_n &=\e \exp\left\{-\beta\sum_{k \le b}\eta_k/\beta_k(n)\right\} \e \sum_{i > b+1} \exp\left\{-\beta \sum_{b < k < i} {\eta_{k}}/{\beta_{k}(n)}\right\}\\
&= \prod_{k \le b}\frac{1}{1+\beta /\beta_{k}(n)}\;\sum_{i> b+1}^n \prod_{b<k < i}\frac{1}{1+\beta /\beta_{k}(n)}\\
&= F_n(b) \cdot G_n(b)\qquad \text{(say)}.
\end{split}
\] 
Now, since $\beta > 1$, by \eqref{limiti}, there exists $n_0\in \mathbb{N}$ such that 
\[
\sup_{n\ge n_0}\max_{1\le i\le n-1} {\beta_i(n)}/{i} < \sqrt{\beta}.
\]
In other words, for all $n \ge n_0$, and for all $i=1,2,\ldots,n-1$, we have $\beta_i(n)<\sqrt{\beta} i$. Thus, for all $n \ge n_0$, we have
\[
\begin{split}
G_n(b)&= \sum_{i > b+1}^n \;\prod_{b < k < i}\frac{1}{1 + \beta/\beta_{k}(n)}\\
&\le \sum_{i > b+1}^{\infty} \; \prod_{b < k < i}\frac{1}{1 + \sqrt{\beta}/k}= \sum_{i > b+1}\;\prod_{b < k < i}\frac{k}{k + \sqrt{\beta}}\\
&= \sum_{i > b+1}\; \frac{(i-1)!\;\Gamma(b+\sqrt{\beta}+1)}{b!\;\Gamma(i+\sqrt{\beta})}. 
\end{split}
\]
By Stirling's approximation to the Gamma function, we can deduce (see, e.g., \cite[page 76, eq. 6.1.2]{handgamma}) that there exists a constant $C$ depending on $\beta$ such that
\[
\frac{(i-1)!}{\Gamma(i+\sqrt{\beta})} \le \frac{C}{i^{\sqrt{\beta}}},\quad \forall\; i \ge 1,
\]
and thus one gets the bound
\[
\begin{split}
\limsup_nG_n(b)&\le C \frac{\Gamma(b+\sqrt{\beta}+1)}{b!}\sum_{i > b+1}{i^{-\sqrt{\beta}}}\\
&\le C \frac{\Gamma(b+\sqrt{\beta}+1)}{b!}\int_{b}^{\infty}\frac{ds}{s^{\sqrt{\beta}}}\; \le\; C'\frac{\Gamma(b+\sqrt{\beta}+1)}{b!\;b^{\sqrt{\beta}-1}}.
\end{split}
\]
where $C'$ is another constant possibly depending on $\beta$.

Similarly, one can bound $F_n(b)$ by
\[
\begin{split}
\limsup_n F_n(b)&= \prod_{k \le b}\frac{1}{1+\beta/\beta_{k}(n)}\\
&\le \prod_{k \le b}\frac{1}{1+\sqrt{\beta}/k}= \frac{b!}{(\sqrt{\beta}+1)(\sqrt{\beta}+2)\ldots(\sqrt{\beta} + b)}.
\end{split}
\]
Thus, the product $\e Q_n(b)$ can be bounded by
\[
\begin{split}
\e Q_n(b) &\le C'\frac{\Gamma(b+\sqrt{\beta}+1)}{b!\;b^{\sqrt{\beta}-1}}\frac{b!}{(\sqrt{\beta}+1)(\sqrt{\beta}+2)\ldots(\sqrt{\beta} + b)}\\
&\le \frac{C'\Gamma(\sqrt{\beta}+1)}{b^{\sqrt{\beta}-1}}=\frac{C_1}{b^{\sqrt{\beta}-1}}.
\end{split}
\]
Here $C_1$ is another constant depending on $\beta$. This proves the first part of the lemma.

The second part of bounding $R(n)$ follows if we define the rates of the random variables $Y'_i(n)$ as an array 
\[
\beta'_k(n)=\begin{cases} 
{\beta_{k}(n)}&,\quad 1\le k \le b,\\
k&,\quad b< k \le n-1.
\end{cases}
\]
Then the array $\beta'_k(n)$ also satisfy \eqref{limiti} and hence from the previous argument it follows that $\e R_n(b) \le C_1 b^{1-\sqrt{\beta}}$, and the lemma is proved.
\end{proof}

We are now ready to finish the first part of Lemma \ref{mainthm}.
We start with a test function on the space $(\mcal{S}',{\bf d}')$, introduced in \eqref{whatisdp}. That is to say a function $f:\mcal{S}'\rightarrow \rr$ which is Lipschitz with Lipschitz coefficient one. Let $PD(f)$ denote the expectation of $f$ with respect to the PD($1/\beta$) law. Now, by triangle inequality,
\begin{equation}\label{convtopd}
\begin{split}
\limsup_{n\rightarrow \infty}\abs{\e f(U_n) - PD(f)}&\le \limsup_n\abs{\e f(U_n) - \e f(\widetilde{U}_n)}\\ 
&+ \limsup_n\abs{\e f(\widetilde{U}_n) - PD(f)}\\
\le \limsup_n \e\;{\bf d}'(U_n,\widetilde{U}_n) &+  \limsup_n\abs{\e f(\tilde{U}_n) - PD(f)},\quad (f-\text{Lipschitz}),\\
&\le \frac{C}{b^{\sqrt{\beta}-1}} + 0.
\end{split}
\end{equation}
The first limit above is by Lemma \ref{approxb} and the second limit is zero by Lemma \ref{convergeetilde}. 

We can now take $b$ to infinity in \eqref{convtopd}, and since $\beta > 1$, we conclude that $\lim_{n \rightarrow \infty}\e f(U_n) = PD(f)$. Since this is true for all Lipschitz continuous functions, the conclusion is now evident from the standard theory of weak convergence. This completes the proof of part (i) of Lemma \ref{mainthm}.

\vskip.2in
Let us now prove the second assertion in Lemma \ref{mainthm}. By Lemma \ref{ppcoupling} this is equivalent to proving that if $0 \le \delta_n \le 1$, and $\lim_n \delta_n = \beta \le 1$, then each coordinate of the point process $U_n(\delta_n)$ converges in probability to zero. Obviously, it suffices to show that the first element of $U_n(\delta_n)$, which is also the maximum among the elements, goes to zero in probability. This will be shown in the rest of this section. Throughout the rest of this section, the parameter $\{\delta_n\}$ will be a sequence such that each $\delta_n$ is positive and less than or equal to one and $\lim_n \delta_n= \beta \in [0,1]$.
\medskip

As usual, let $\mu_1(n)$ denote the largest element in the random sequence $U_n(\delta_n)$. 
Then, we have
\begin{equation}\label{whatismun}
\begin{split}
\mu_1(n) 
&= \frac{1}{1 + \sum_{j=2}^{n} \exp(-\delta_n \sum_{i=1}^{j-1} \xi_{i}(n)   )}
\end{split}
\end{equation}
The following lemma, which proof will follow later, is our key tool in proving that $\mu_1(n) \rightarrow 0$.


\begin{lem}\label{phase}
Suppose $V_1,\ldots,V_{K}$ are independent Exponential random variables, with $V_i \sim \text{Exp}(\lambda_i)$. Let $\theta_i = 1/\lambda_i = E(V_i)$. Let
\begin{equation}\label{muv}
\mu_V = \frac{1}{1 + e^{-V_1} + e^{-(V_1+V_2)} +\cdots + e^{-(V_1+\cdots+V_{K})}}.
\end{equation}
Let us also define the quantities
\begin{equation}\label{mubarv}
\bar{\mu}_V := \frac{1}{1 + e^{-\theta_1} + e^{-(\theta_1+\theta_2)} +\cdots + e^{-(\theta_1+\cdots+\theta_{K})}},\;\;\sigma := \bigl(\sum_{i=1}^{K} \theta_i^2\bigr)^{1/2}.
\end{equation}

We have 
\begin{eqnarray}
E(\mu_V) &\ge& e^{-\sigma^2} \bar{\mu_V},\quad \text{and}\label{mulbnd}\\
E(\mu^{1/2\sigma}_V) &\le& 4e^{1/4} \bar{\mu}_V^{1/2\sigma}.\quad \text{Moreover},\label{muubnd}\\
E(\log \mu_V - \log \bar{\mu}_V)^2 &\le& 8\sigma^2.\label{mupcare}
\end{eqnarray}
\end{lem}

Let us now complete the proof of Lemma \ref{mainthm} with the aid of the last lemma. Fix some $K\ge 2$ and define
\begin{equation}\label{whatismuk}
\mu_{1,K}(n) := \frac{1}{1 + \sum_{j=2}^K \exp\{-\delta_n \sum_{i=1}^{j-1} \xi_{i}(n)\}} \ge \mu_1(n).
\end{equation}
In the following, we will simply write $\mu_n$ and $\mu_{K,n}$ instead of $\mu_1(n)$ and $\mu_{1,K}(n)$ respectively. 
For fixed $n$, if we take $\theta_i(n)=\delta_n/\beta_{i}(n)$, and consider $V_i$ as Exponential with mean $\theta_i(n)$ for $i=1,\ldots,K$, then it is clear that $\mu_V$ has the same law $\mu_{K,n}$. 
Analogously, let
\[
\bar{\mu}_{K,n} := \frac{1}{1 + \sum_{j=2}^K \exp\{-\sum_{i=1}^{j-1} \theta_{i}(n)\}}.
\]
Thus, for every $\epsilon >0$, we can apply Lemma \ref{phase} to get the Markov's bound
\begin{equation}\label{markov}
P(\mu_n > \epsilon ) \le P(\mu_{K,n} > \epsilon)\le \frac{\e \mu^{1/2\sigma}_{K,n}}{\epsilon^{1/2\sigma}} \le C\frac{\bar{\mu}_{K,n}^{1/2\sigma}}{\epsilon^{1/2\sigma}}.
\end{equation}
Here $C$ is the constant $4e^{1/4}$, and $\sigma$ (depends on $K$ and $n$) is given by
\[
\sigma^2(K,n)= \sum_{i=1}^K \left( \frac{\delta_n}{\beta_{i}(n)} \right)^2.
\]

It follows from condition \eqref{limiti}, and since $\delta_n \le 1$, that as $n$ grows to infinity, while keeping $K$ fixed, $\sigma$ has a limit given by 
\[
a_K=\lim_{n\rightarrow \infty}\sigma^2(K,n)= \beta^2\sum_{i=1}^K \frac{1}{i^2}.
\]
Thus, taking limit as $n$ tends to infinity while keeping $K$ fixed in \eqref{markov}, we obtain
\begin{equation}\label{limit1}
\limsup_{n\rightarrow \infty} P(\mu_n > \epsilon) \le C\left(\frac{1}{\epsilon}\limsup_n \bar{\mu}_{K,n}\right)^{1/2a_K}.
\end{equation}
The RHS of \eqref{limit1} will be interpreted as zero if $\beta$ is zero and $\limsup_n \bar{\mu}_{K,n} < \epsilon$.

Suppose now that $\beta > 0$. In the next lemma, Lemma \ref{finalzero}, we will show that 
\begin{equation}\label{lastkey}
\lim_{K\rightarrow \infty}\limsup_{n\rightarrow\infty} \bar{\mu}_{K,n}=0.
\end{equation}
Since $a_K$ grows to a finite limit $\beta^2\pi^2/6$ as $K$ tends to infinity, we take a further limit in \eqref{limit1} to get
\[
\begin{split}
\limsup_{n\rightarrow \infty} P(\mu_n > \epsilon) &\le \lim_{K\rightarrow \infty}C\left(\frac{1}{\epsilon}\limsup_n \bar{\mu}_{K,n}\right)^{1/2a_K}\\
&= C\left(\frac{1}{\epsilon}\lim_K\limsup_n \bar{\mu}_{K,n}\right)^{3/\beta^2\pi^2}=0.
\end{split}
\]
This proves that $\mu_n$ goes to zero in probability.

If $\beta=0$, for every $\epsilon >0$, Lemma \ref{finalzero} shows that there exists $K$ such that 
\[
\limsup_{n\rightarrow \infty} \bar{\mu}_{K,n} < \epsilon.
\]
We apply \eqref{limit1} to this $K$ and obtain that $\limsup_n P(\mu_n > \epsilon) =0$. Hence we have established the second claim in Lemma \ref{mainthm} completing its proof.

\begin{lem}\label{finalzero}
\[
\lim_{K\rightarrow \infty}\limsup_{n\rightarrow\infty} \bar{\mu}_{K,n}=0.
\]
\end{lem}

\begin{proof}
We can write down $\bar{\mu}_K$ exactly in terms of the $\theta_i(n)=\delta_n /{\beta_{i}(n)}$ as
\[
\bar{\mu}_{K,n}= \frac{1}{1 + e^{-\theta_1} + e^{-(\theta_1+\theta_2)} +\cdots + e^{-(\theta_1+\cdots+\theta_{K})}}.
\]
As we take $n$ to infinity keeping $K$ fixed, it follows from condition \eqref{limiti} that
\[
\begin{split}
\bar{\mu}_K:=\limsup_n \bar{\mu}_{K,n}&= \frac{1}{1 + e^{-\beta/1} + e^{-(\beta/1+\beta/2)} +\cdots + e^{-(\beta/1+\beta/2+\cdots+\beta/{K})}}\\
&\le  \frac{\text{const.}}{1 + e^{-\beta} + e^{-(\beta \log 2)} +\cdots + e^{-(\beta\log K)}}\\
&\le \frac{\text{const.}}{1 + 2^{-\beta} +\cdots + K^{-\beta}}.
\end{split}
\]
Note that since $\beta \le 1$, the denominator of the last expression above goes to infinity as $K$ tends to infinity. Thus
\[
\lim_{K\rightarrow \infty}\bar{\mu}_K = \text{const.}\lim_{K\rightarrow\infty}\left(1 + \frac{1}{2^{\beta}} + \ldots + \frac{1}{K^{\beta}} \right)^{-1}=0.
\]
This proves the lemma.
\end{proof}

\begin{proof}[Proof of Lemma \ref{phase}]
First, by Jensen's inequality, we have
\begin{align*}
E(\mu_V) &\ge \frac{1}{1 + E(e^{-V_1}) + E(e^{-(V_1+V_2)}) +\cdots + E(e^{-(V_1+\cdots+V_{K})})}.
\end{align*}
Now, by independence, $E(e^{-(V_1+\cdots+V_i)}) = \prod_{j=1}^i (1+\theta_j)^{-1}$.
Thus,
\begin{equation}\label{lbd1}
E(\mu_V)\ge \biggl(1 + \sum_{i=1}^{K} \prod_{j=1}^i \frac{1}{1+\theta_j}\biggr)^{-1}.
\end{equation}
Now, the function $\log (1+x) - (x-x^2)$ can be easily verified to be zero at zero, and increasing on the positive half-line. Thus,
\[
\frac{1}{1+\theta_j} \le e^{-\theta_j +\theta_j^2}.
\]
It follows that
\begin{equation}\label{lbd2}
\begin{split}
E(\mu_V) &\ge \frac{1}{1 + \sum_{i=1}^K \prod_{j=1}^i e^{-\theta_j + \theta_j^2}}=\frac{1}{1+\sum_{i=1}^K  e^{-\sum_1^i\theta_j+\sum_1^i\theta_j^2}}\\
&\ge \frac{\exp\{-\sum_{j=1}^{K} \theta_j^2\}}{ 1+\sum_{i=1}^{K} \exp\{-\sum_{j=1}^i \theta_i\}} = e^{-\sigma^2} \bar{\mu}_V.
\end{split}
\end{equation} 
This proves \eqref{mulbnd}.

Next, consider the martingale 
\[
M_i = \sum_{j=1}^i (V_j - \theta_j),\quad i=1,\ldots,K.
\]
For the non-negative submartingale $\exp(M_i/4\sigma)$ and any $p > 1$, we apply Doob's $L^p$-inequality (\cite[page 54]{ry99}) to get
\begin{equation}\label{dooblp}
E\left(\max_{1\le i\le K} e^{pM_i/4\sigma}\right) \le \left(\frac{p}{p-1}\right)^p E(e^{pM_{K}/4\sigma}).
\end{equation}
Now, by definition, $\theta_i \le \sigma$ for every $i$. Thus, with $p=2$, we get
\begin{align*}
E(e^{M_{K}/2\sigma}) &= \prod_{i=1}^{K} E(e^{(V_i-\theta_i)/2\sigma})
= \prod_{i=1}^{K} \frac{e^{-\theta_i/2\sigma}}{1-\frac{\theta_i}{2\sigma}}.
\end{align*}
Now, as in the case of the upper bound, it is straightforward to verify that for any $x\in [0,1/2]$, $1-x \ge e^{-x-x^2}$. Thus, plugging in this inequality in the last expression we get
\begin{equation}\label{estmk}
E(e^{M_{K}/2\sigma}) \le e^{\sum_i \theta_i^2/4\sigma^2} = e^{1/4}. 
\end{equation}
Finally, note that for any two positive sequences $\{x_i\}$ and $\{y_i\}$ one has
\[
\frac{x_1+x_2+\ldots+x_K}{y_1+y_2+\ldots+ y_K} \le \max_{1\le i\le K}\frac{x_i}{y_i}.
\]
We use the above inequality to obtain
\begin{align}
\mu_V &= \frac{1 + e^{-\theta_1} + \cdots + e^{-(\theta_1+\cdots+\theta_{K})}}{1+ e^{-V_1} + \cdots + e^{-(V_1+\cdots+V_{K})}}\biggl(1+\sum_{i=1}^K e^{-\sum_{j=1}^i \theta_i}\biggr)^{-1}\nonumber \\
&\le \max_{1\le i\le K} e^{M_i}\cdot \bar{\mu}_V.\label{ratioseq}
\end{align}
Combining \eqref{dooblp}, \eqref{estmk}, and \eqref{ratioseq}, we get
\[
E(\mu_V^{1/2\sigma})\le 4e^{1/4} \bar{\mu}_V^{1/2\sigma},
\]
which proves \eqref{muubnd}.

Next, from \eqref{ratioseq} observe that
\begin{equation}\label{bias}
|E(\log \mu_V) - \log \bar{\mu}_V| \le E\bigl(\max_{1\le i\le K} \abs{M_i}\bigr) \le 2[E(M_K^2)]^{1/2} = 2\sigma.
\end{equation}
Now recall that the Exponential distribution satisfies the Poincar\'e inequality with Poincar\'e constant $4$ (see e.g.\ \cite{bobkov99}, p.\ 2). That is, for any function $f(V_1,\ldots,V_{K})$, 
\begin{equation}\label{varbnd}
\text{Var}(f(V_1,\ldots,V_{K})) \le 4\sum_{i=1}^{K} \theta_i^2 E\biggl(\frac{\partial f}{\partial x_i}(V_i)\biggr)^2.
\end{equation}
If we define
\[
f(x_1,x_2,\ldots,x_K)= -\log\;\left[{1+ \sum_{j=1}^{K} \exp  \left\{-\sum_{i=1}^{j} x_i\right\}}\right], 
\] 
note that $f(V_1,V_2,\ldots,V_K)=\log \mu_V$. The partial derivatives of $f$ are given by
\[
\frac{\partial f}{\partial x_i} = \frac{\sum_{j=i}^{K} \exp\left\{-\sum_{k=1}^{j} x_k\right\}}{1+ \sum_{j=1}^{K} \exp\left\{-\sum_{k=1}^{j} x_k \right\}} \in [0,1].
\]

Thus, a variance upper bound follows from \eqref{varbnd}: $\text{Var}(\log \mu_V) \le 4 \sigma^2$. Combining the last step with equation \eqref{bias}, we get
\[
E(\log \mu_V - \log \bar{\mu}_V)^2 \le 8\sigma^2.
\]
This completes the proof of the theorem. 
\end{proof}

\subsection{Proof of Theorem \ref{phasegrav}.}

\begin{proof}[Proof of Theorem \ref{phasegrav}] 
In accordance to the set-up in Lemma \ref{phase}, 
we define
\[
\lambda_i(n)  = 2\sum_{j=1}^i (\bar{\delta}(n) - \delta_i(n)),
\]
and
\[
\theta_i(n) = \frac{1}{\lambda_i(n)}, \ i = 1,\ldots,n-1.
\]
Now let $Y_1(n),\ldots,Y_{n-1}(n)$ denote the successive increments in the stationary distribution (as in Theorem \ref{mainthmdrift}). Then $Y_i(n)$ has an Exponential distribution with rate $\lambda_i(n)$. 
Let $\gamma' = (\eta /2C) \wedge \gamma$, and define
\begin{align*}
\mu_1'(n) &:= \frac{1}{1 + \sum_{i=1}^{\lfloor \gamma' n\rfloor} \exp(-\sum_{j=1}^i Y_i(n))}, \ \ \text{and}\\
\bar{\mu}_1'(n) &:= \frac{1}{1 + \sum_{i=1}^{\lfloor \gamma' n\rfloor} \exp(-\sum_{j=1}^i \theta_i(n))}.
\end{align*}
Now let 
\[
\epsilon_n = |2(\bar{\delta}(n) - \delta_1(n)) - 2\eta|,
\]
and assume that $n$ is large enough to ensure that $\epsilon_n \le \eta/2$. Since $|\delta_1(n) - \delta_i(n)|\le C(i-1)/n$, it follows that
\[
|\lambda_i(n) - 2i\eta|\le \epsilon_n i + 2C\sum_{j=1}^i \frac{i-1}{n} \le \epsilon_n i +  \frac{C i^2}{n}, \ \ i=1,\ldots,n - 1.
\]
If $i\le \gamma' n$, it follows in particular that
\begin{equation}\label{firstineq}
\lambda_i(n) \ge 2i\eta - \epsilon_n i - Ci (\eta/2C)\ge \eta i.
\end{equation}
Thus, for $n$ sufficiently large and $i\le \gamma' n$,
\begin{align}\label{mainineq}
\biggl|\theta_i(n) - \frac{1}{2i\eta} \biggr| &= \biggl|\frac{\lambda_i(n) - 2i\eta}{2i\eta \lambda_i(n)}\biggr|\le \frac{\epsilon_n}{i} + \frac{C}{2\eta^2 n}.
\end{align}
Now let
\[
\alpha_n^+ = \frac{1}{2\eta} + \epsilon_n.
\]
Then for $n$ sufficiently large and $i\le \gamma' n$,
\[
\sum_{j=1}^i \theta_j(n) \le \alpha_n^+ \sum_{j=1}^i \frac{1}{j} + \frac{Ci}{2\eta^2 n} \le \alpha_n^+ \log i + L(\eta),
\]
where $L(\eta)$ does not depend on $n$. Now, if $\eta > 1/2$ then for all sufficiently large $n$, $\alpha_n^+ < (1 + 1/2\eta)/2 < 1$. Thus,
\begin{equation}\label{bd2}
\begin{split}
\bar{\mu}'_1(n) 
&\le \frac{1}{1+  e^{-L(\eta)}\sum_{i=1}^{\lfloor \gamma'n\rfloor } i^{-\alpha_n^+}}\le \kappa^+(\eta) n^{\alpha_n^+ - 1},
\end{split}
\end{equation}
where $\kappa^+(\eta)$ is a constant free of $n$. Another application of \eqref{mainineq} shows that
\[
\bar{\mu}'_1(n) \ge \kappa^-(\eta) n^{\alpha_n^- - 1}
\]
where $\alpha_n^- = (1/2\eta) - \epsilon_n$ and $\kappa^-(\eta)$ is another constant depending only on $\eta$. Since $\lim \alpha_n^+ = \lim \alpha_n^- = 1/2\eta$, this shows that when $\eta > 1/2$, 
\[
\lim_{n \rightarrow \infty} \frac{\log \bar{\mu}_1'(n)}{\log n}  = \frac{1}{2\eta} - 1.
\]
Again, by \eqref{firstineq} it follows that
\begin{equation}\label{l2}
\limsup_{n\rightarrow\infty } \sum_{i=1}^{\lfloor\gamma'n\rfloor} \theta_i(n)^2 < \infty.
\end{equation}
We can now use Lemma \ref{phase} to conclude that
\[
\frac{\log \mu_1'(n)}{\log n}  \stackrel{P}{\rightarrow} \frac{1}{2\eta} - 1.
\]
Finally, note that
\begin{align*}
|\log \mu_1(n) - \log \mu_1'(n)| &= \log \biggl( 1 + \frac{\sum_{i=\lfloor \gamma' n\rfloor + 1}^{n-1} \exp(-\sum_{j=1}^i Y_j(n))}{1 + \sum_{i=1}^{\lfloor \gamma' n\rfloor} \exp(-\sum_{j=1}^i Y_j(n))}\biggr)\\
&\le \frac{\sum_{i=\lfloor \gamma' n\rfloor + 1}^{n-1} \exp(-\sum_{j=1}^i Y_j(n))}{1 + \sum_{i=1}^{\lfloor \gamma' n\rfloor} \exp(-\sum_{j=1}^i Y_j(n))}\\
&\le \frac{\sum_{i=\lfloor \gamma' n\rfloor + 1}^{n-1} \exp(-\sum_{j=1}^i Y_j(n))}{(\lfloor \gamma' n\rfloor + 1) \exp(-\sum_{j=1}^{\lfloor \gamma' n\rfloor}  Y_j(n))}\\
&= \frac{1}{\lfloor \gamma' n\rfloor + 1} \sum_{i=\lfloor \gamma' n\rfloor + 1}^{n-1} \exp\biggl(-\sum_{j=\lfloor \gamma' n\rfloor + 1}^i Y_j(n)\biggr)\\
&\le \frac{n - \lfloor \gamma' n\rfloor - 1}{\lfloor \gamma' n \rfloor+ 1}\le \frac{1-\gamma'}{\gamma'}.
\end{align*}
This completes the proof of the Theorem in the case $\eta > 1/2$.

Now, in the case $ \eta = 1/2$, everything up to the first inequality in \eqref{bd2} is still valid. However, the second inequality does not hold since $\lim \alpha_n^+ = 1$. Instead, we use the inequality
\[
\sum_{i=1}^{\lfloor \gamma' n\rfloor} i^{-\alpha_n^+} \ge n^{-\epsilon_n} \sum_{i=1}^{\lfloor \gamma' n\rfloor} i^{-1} \ge Kn^{-\epsilon_n} \log n
\]
and the assumption that $\epsilon_n = O(1/\log n)$ to conclude that
\[
\bar{\mu}_1'(n) \le \frac{K}{ \log n},
\]
where $K$ is some constant that does not depend on $n$. The inequality in the opposite direction follows similarly, possibly with a different constant. Inequality~\eqref{l2} continues to hold without any problem, and so does the subsequent argument. Combining, it follows that 
\[
\frac{\log \mu_1(n)}{\log \log n} \stackrel{P}{\rightarrow} -1.
\]
This completes the proof.
\end{proof}

\subsection{Proof of Theorem \ref{benefits}.}

\begin{proof}[Proof of Theorem \ref{benefits}]
The central argument in this proof is to apply continuous mapping theorem to the weak convergence result in Theorem \ref{mainthmdrift}. Let $Q_n$ be the stationary law of the market weights $\{\mu_1(n), \mu_2(n),\ldots, \mu_n(n)\}$ arranged as a decreasing sequence on the line. Then, by Theorem \ref{mainthmdrift}, $Q_n$ converges weakly to the PD($2\eta$) law as $n$ tends to infinity. Let $\{ c_i,\; i=1,2,\ldots\}$ denote a random sequence whose law is PD($2\eta$).

To prove (i), note that the function $\mathbf{x}\mapsto x_1^p$ is a continuous map in the space of sequences $(\mcal{S}', \mathbf{d}')$ (see \eqref{whatisdp}). It follows that $\mu_1^p(n)$ converges weakly to the law of $c_1^p$, and hence
\[
\begin{split}
\lim_{n\rightarrow \infty}\e \mu_1^p(n) &= \e c_1^p = \frac{1}{\Gamma(p)}\int_0^{\infty}t^{p-1}e^{-t}\psi^{-1}_{2\eta}(t)dt,\\
\psi_{2\eta}(t)&=1 + 2\eta\int_0^{1}(1-e^{-tx})x^{-2\eta-1}dx,
\end{split}
\]
which proves \eqref{moment}. The expressions for the moment of the coordinates of a PD random sequence can be found in \cite[Proposition 17]{poidir97} (for $\theta=0, \alpha=2\eta, n=1$).

Proving (ii) is similar. Again, by the continuous mapping theorem we know that for each $i$, 
\[
\lim_{n\rightarrow \infty} \e \mu_i^p(n) = \e c_i^p.
\]
Now, for any bounded function $f$ on $[0,1]$, the following identity that can be found in \cite[page 858]{poidir97}:
\[
\e \sum_{i=1}^\infty f(c_i) = \frac{1}{\Gamma(2\eta)\Gamma(1-2\eta)}\int_0^1 f(u) \frac{(1-u)^{2\eta -1}}{u^{2\eta +1}}du.
\]
If we take $f(x)=x^p$, as before, it is easy to verify that
\begin{equation}\label{cevans}
\e \sum_{i=1}^{\infty} c_i^p =  \frac{\Gamma(p-2\eta)}{\Gamma(p)\Gamma(1-2\eta)}.
\end{equation}
Now, by the dominated convergence theorem, it is clear that $\lim_n \e\sum \mu_i^p(n)$ equals the above expression if we can find a function $\psi: \mathbb{Z}_+ \rightarrow \mathbb{R}$ such that (i) $\sum_{i=1}^\infty \psi(i) < \infty$,  and (ii) for all sufficiently large $n$, $\e\mu_i^p(n) \le \psi(i)$ for all $i$.

Now, in the notation of Theorem \ref{mainthmdrift}, we have
\[
\mu_i(n) = \frac{e^{X_i(n) - X_1(n)}}{\sum_{j=1}^n e^{X_j(n) - X_1(n)}} \le e^{X_i(n) - X_1(n)} = e^{-\sum_{j=1}^{i-1} Y_j(n)},
\]
where $Y_j(n)$'s are independent Exponentials with rate $2\alpha_j(n)$ from \eqref{conal}. Now fix $p' \in (2\eta, p)$. Let $r_n = \max_{1\le i\le n} (\bar{\delta}(n) - \delta_i(n))$. Since $\lim_n r_n = 2\eta$, we can assume that $n$ is sufficiently large to ensure that $r_n \le p'$. Then
\[
\e(\mu_i^p(n)) \le \prod_{j=1}^{i-1} \frac{1}{1+p/2\alpha_j(n)}\le \prod_{j=1}^{i-1} \frac{1}{1+p/jp'} =: \psi(i).
\]
It is not difficult to see that for large $i$, $\psi(i)$ is comparable with $i^{-p/p'}$. Since $p' < p$, it follows that $\sum \psi(i) < \infty$, and this completes the proof of part (ii).

For part (iii), let us begin with the definitions
\begin{equation}\label{whatishp}
h_n(p) = \e \sum_{i=1}^n \mu_i^p(n),\quad h(p)= \e \sum_{i=1}^{\infty} c_i^p, \quad p > 2\eta. 
\end{equation}
By part (ii), $h_n(p)$ converges to $h(p)$ as $n$ tends to infinity pointwise for every $p > 2\eta$. Also, note that $\{h_n\}$ and $h$ are differentiable convex functions in $(2\eta,\infty)$. It follows that (see, e.g., \cite[Proposition 4.3.4]{canamial}) their derivatives also converge pointwise in the relative interior $(2\eta, \infty)$. That is to say, $\lim_{n\rightarrow \infty} h_n'(p) = h'(p)$. Taking derivatives inside the expectation at $p=1$ in \eqref{whatishp}, we get
\[
\lim_{n\rightarrow \infty} \e \sum_{i=1}^n \mu_i(n)\log \mu_i(n) = h'(1)=\e \sum_{i=1}^{\infty} c_i \log c_i.
\]
Now, to evaluate the last expression, we use the expression of $h$ from \eqref{cevans}:
\[
h'(p)= \frac{1}{\Gamma(1-2\eta)}\left[ \frac{\Gamma'(p-2\eta)}{\Gamma(p)} - \frac{\Gamma(p-2\eta)\Gamma'(p)}{\Gamma^2(p)}\right].
\]
Thus, 
\begin{equation}\label{limneg}
\e \sum_{i=1}^{\infty} c_i \log c_i= \frac{\Gamma'(1-2\eta)}{\Gamma(1-2\eta)} - \frac{\Gamma'(1)}{\Gamma(1)}.
\end{equation}
It is known that (see, for example \cite[6.3.2, 6.3.16, page 79]{handgamma})
\[
\frac{\Gamma'(1)}{\Gamma(1)}= -\gamma,\quad \frac{\Gamma'(1-2\eta)}{\Gamma(1-2\eta)}= - \gamma - 2\eta\sum_{k=1}^{\infty} \frac{1}{k(k-2\eta)},
\]
where $\gamma$ is the Euler constant. Plugging in these expressions in \eqref{limneg}, we get
\[
\lim_{n\rightarrow \infty}\e \left[-\sum_{i=1}^n \mu_i(n)\log \mu_i(n)\right] = 2\eta \sum_{k=1}^{\infty} \frac{1}{k(k-2\eta)},
\] 
which proves the result.
\end{proof}

\section{Counterexamples}\label{examples}

We now provide two examples where Theorem \ref{mainthmdrift} fails due to the absence of one of the two conditions \eqref{conddel2} and \eqref{conddel}.
\medskip

\noindent\textbf{Example.} Let us consider SDE \eqref{ranksde} for $n$ particles where
\[
\delta_1(n)= -1/4, \quad \delta_i(n)=0,\quad i=2,3,\ldots 
\]
This is like the Atlas model with parameter $1/4$ except that the push comes from the top. Thus
\begin{equation}\label{exmalpha}
\bar{\delta}(n)=-\frac{1}{4n}, \quad \alpha_k(n) = \sum_{i=1}^k (\bar{\delta}(n) - \delta_i(n))=-\frac{k}{4n}+\frac{1}{4}=\frac{n-k}{4n},   
\end{equation}
which is positive for $1\le k \le n-1$. Thus, the drift sequence satisfy condition \eqref{conal}. Additionally the drift sequence satisfy \eqref{conddel} for $\eta=1/4$, since
\[
\lim_{n\rightarrow\infty}\max_{1\le i \le n}(\bar{\delta}(n) - \delta_i(n))=\frac{1}{4}-\lim_{n\rightarrow\infty}\frac{1}{4n}=\frac{1}{4}\in (0,1/2).
\]
However, the drift sequence does not satisfy \eqref{conddel}.

It is easy to see that the market weights for this model do not converge to any Poisson-Dirichlet law. In fact, consider the following ratios  
\[
\frac{\mu_2(n)}{\mu_1(n)}= e^{-\xi_1(n)},\quad \frac{\mu_3(n)}{\mu_2(n)}=e^{-\xi_2(n)}
\]
where $\xi_i(n)$ is Exponential with rate $\alpha_i(n)$ for $i=1,2$. From \eqref{exmalpha}, one can verify that $\lim_{n\rightarrow\infty}\alpha_1(n)=\lim_{n\rightarrow\infty}\alpha_2(n)=1/4$. Thus, it follows that the asymptotic laws of $\mu_2(n)/\mu_1(n)$ and $\mu_3(n)/\mu_2(n)$ are the same. However, if $(V_1,V_2,V_3,\ldots)$ is a random sequence from any PD distribution, it is known that (see Proposition 8 in \cite{poidir97}) the laws of $V_2/V_1$ and $V_3/V_2$ are different Beta random variables. This negates the possibility that the market weights converge weakly to any PD law.
\medskip

\noindent\textbf{Example.} For the second example, consider any $\eta \in (0,1/2)$. Let $\beta = 4(1-\eta)$. For each $n$, define $\delta_i(n), \ i = 1,\ldots, n$ as follows: 
\[
\delta_i(n)=\begin{cases}
-\eta,& \qquad 1\le i\le \flr{n^\eta},\\
-\beta, & \qquad \flr{n^\eta} + 1 \le i\le \flr{n/2},\; \text{and}\\
- \delta_{n-i+1}(n), &\qquad i > \flr{n/2}.
\end{cases}
\]
Then, by symmetry, $\bar{\delta}(n) = 0$ and the drifts satisfy the condition \eqref{conal}. 

It is also clear that the drifts satisfy \eqref{conddel2}, but \eqref{conddel} does not hold. We shall show that $\mu_1(n) \rightarrow 0$ even though $\eta \in (0,1/2)$.

Let $Y_i(n), \ i=1,\ldots,n-1$ be the successive spacings under the stationary law as in Theorem \ref{theoremN}. Then $Y_i(n)$ is Exponential with rate
\[
\lambda_i(n) = -2\sum_{j=1}^i \delta_j(n).
\]
Note that by our construction, $\lambda_{n-i+1} = \lambda_i$. Again, since $\beta > \eta$, it follows that for $1\le i\le 2\flr{n^\eta}$ we have
\begin{equation}\label{lambda1}
\lambda_i(n) \ge 2\eta i.
\end{equation}
If $2\flr{n^\eta} < i\le \flr{n/2}$, then $i - \flr{n^{\eta}} \ge i/2$, and we get
\begin{equation}\label{lambda2}
\lambda_i(n) = 2\eta \flr{n^\eta} + 2\beta(i-\flr{n^\eta}) \ge \beta i.
\end{equation}
Now let $\theta_i(n) = 1/\lambda_i(n)$. From the above observations and the summability of $\sum_{i=1}^\infty i^{-2}$, it is clear that $\sum_{i=1}^{n-1}\theta_i^2(n)$ can be bounded by a constant that does not depend on $n$. So by Lemma \ref{phase}, we can conclude that $\mu_1(n) \rightarrow 0$ in this model provided we can show that $\bar{\mu}_1(n) \rightarrow 0$, where
\[
\bar{\mu}_1(n) = \frac{1}{1 + \sum_{i=1}^{n-1} \exp(-\sum_{j=1}^i \theta_j(n))}.
\]
Note that
\begin{equation}\label{inequ}
\begin{split}
\bar{\mu}_1(n) &\le \frac{1}{\sum_{i=2\flr{n^\eta}+1}^{\flr{n/2}} \exp(-\sum_{j=1}^i \theta_j(n))} \\
&= \frac{\exp(\sum_{j=1}^{2\flr{n^\eta}} \theta_j(n))}{\sum_{i=2\flr{n^\eta}+1}^{\flr{n/2}} \exp(-\sum_{j=2\flr{n^\eta} + 1}^i \theta_j(n))}.
\end{split}
\end{equation}
By \eqref{lambda1}, we get
\begin{equation}\label{inequ1}
\sum_{j=1}^{2\flr{n^\eta}} \theta_j(n) \le \frac{1}{2\eta}\sum_{j=1}^{2\flr{n^\eta}} \frac{1}{j} \le \frac{\log n}{2} + C,
\end{equation}
for some constant $C$ that does not depend on $n$. Again, for $2\flr{n^\eta} + 1\le i\le \flr{n/2}$, inequality \eqref{lambda2} gives
\begin{equation*}\label{inequ0}
\sum_{j=2\flr{n^\eta} + 1}^i \theta_i(n) \le \frac{1}{\beta} \sum_{j=2\flr{n^\eta} + 1}^i \frac{1}{j} \le \frac{\log(i/n^\eta)}{\beta} + C,
\end{equation*}
where $C$ is again a constant that does not vary with $n$. Thus,
\begin{equation}\label{inequ2}
\begin{split}
\sum_{i=2\flr{n^\eta}+1}^{\flr{n/2}} \exp(-{\textstyle\sum_{j=2\flr{n^\eta} + 1}^i} \theta_i(n)) &\ge e^C \sum_{i=2\flr{n^\eta}+1}^{\flr{n/2}} (i/n^\eta)^{-1/\beta} \\
&\ge C'n^{1 - \frac{1-\eta}{\beta}},
\end{split}
\end{equation}
where $C'$ is some other constant.
Combining the inequalities \eqref{inequ}, \eqref{inequ1}, and \eqref{inequ2}, and using the relation $\beta = 4(1-\eta)$, we get
\[
\bar{\mu}_1(n) \le K n^{\frac{1}{2} -1 + \frac{1-\eta}{\beta}} = K n^{-1/4}.
\]
This shows that $\mu_1(n) \rightarrow 0$ in probability. This shows that market weights do not converge weakly to any PD law inspite of condition \eqref{conddel2} holding with $\eta \in (0,1/2)$.

\bibliographystyle{amsalpha}

\end{document}